\documentclass[12pt,reqno,final]{amsart}

\usepackage{amssymb,amsmath,amsthm}
\usepackage{mathrsfs}
\usepackage[margin=30mm]{geometry}
\usepackage{xcolor}
\usepackage{mathtools}
\mathtoolsset{showonlyrefs}
\usepackage{multirow}
\usepackage[all]{xy}
\usepackage{array}
\usepackage{booktabs}
\usepackage{here}
\usepackage{caption}

\numberwithin{equation}{section}
\usepackage{ascmac}

\theoremstyle{definition}
\newtheorem{thm}{Theorem}[section]
\newtheorem{pro}[thm]{Proposition}
\newtheorem{lem}[thm]{Lemma}

\newtheorem{fact}[thm]{Fact}

\theoremstyle{definition}
\newtheorem{dfn}{Definition}
\newtheorem{ex}{Example}
\newtheorem{rem}{Remark}

\newtheorem{notation}{Notation}

\newcommand{\bs}[1]{{\boldsymbol #1}}
\newcommand{\INN}[2]{\langle #1,#2 \rangle}

\title[Backward mean curvature flow for cohomogeneity one Hermann action]{Cohomogeneity one Hermann actions and backward mean curvature flows}

\author{Kurando Baba}
\address{(K.B.) Department of Mathematics, Faculty of Science and Technology, Tokyo University of Science, 2641 Yamazaki, Noda, 278-8510, Chiba, Japan\\
Research Institute for Science and Technology at Tokyo University of Science,
Division of Joint Research of Geometry and Natural Science\\
Institut f\"ur Mathematik, Universit\"at Augsburg, 86135 Augsburg, Germany
}
\email{kurando.baba@rs.tus.ac.jp}
\thanks{The first author was partially supported by JSPS KAKENHI Grant Number 25K07017.}

\author{Naoyuki Koike}
\address{(N.K.) Department of Mathematics, Faculty of Science, Tokyo University of Science, 1-3 Kagurazaka, Shinjuku-ku, 162-8601, Tokyo, Japan\\
Research Institute for Science and Technology at Tokyo University of Science,
Division of Joint Research of Geometry and Natural Science}
\email{koike@rs.tus.ac.jp}
\thanks{The second author was partially supported by JSPS KAKENHI Grant Number 26K06815.}

\date{\today}



\begin{document}
\maketitle

\begin{abstract}
Principal orbits of cohomogeneity one Hermann actions
form a natural class of equifocal hypersurfaces in symmetric spaces of
compact type.
Based on Kollross's classification of cohomogeneity one actions on irreducible
symmetric spaces of compact type,
we obtain the time evolution
of the squared norm of the
shape operator and the mean curvature along the long-time solution of the
backward mean curvature flows starting from their principal orbits.
These results explicitly describe
the authors' previous generalization of the results of Liu-Terng on backward
mean curvature flows of isoparametric hypersurfaces
to this class of equifocal hypersurfaces.
\end{abstract}

\section{Introduction}
Let $\{f_{t}\}_{t\in[0,T)}$ be a $C^{\infty}$-family of $C^{\infty}$-immersions of a manifold $M$
into a Riemannian manifold $N$,
where $T$ is a positive real number or $T=\infty$.
Let $\bs{H}(t)$ denote the mean curvature vector field of $f_{t}:M\to N$.
The family $\{f_{t}\}_{t\in[0,T)}$ is called a \emph{mean curvature flow}
(abbreviated as MCF)
if the map $F:M\times[0,T)\to N$ defined by $F(x,t)=f_{t}(x)$ satisfies
\begin{equation}
\dfrac{\partial F}{\partial t}=\bs{H}(t).
\end{equation}
Also, the family $\{f_{t}\}_{t\in[0,T)}$ is called a
\emph{backward mean curvature flow}
(abbreviated as BMCF)
if the map $F:M\times[0,T)\to N$ satisfies
\begin{equation}
\dfrac{\partial F}{\partial t}=-\bs{H}(t).
\end{equation}
In particular,
if, for each $t\in[0,T)$, $f_{t}$ is embedding,
then we call $\{M_{t}:=f_{t}(M)\}_{t\in[0,T)}$
rather than $\{f_{t}\}_{t\in[0,T)}$
a MCF or a BMCF.
Liu-Terng (\cite{LiuTe09,LiuTe20}) studied MCFs and BMCFs
starting from isoparametric submanifolds in Euclidean spaces or in spheres.
In particular, they proved that the BMCF starting from
an isoparametric submanifold has a long-time solution.
Furthermore, in the case where the initial submanifold is an isoparametric
hypersurface in a sphere, they gave an explicit description of
the squared norm of the
shape operator and the mean curvature along the long-time solution in the backward direction.
They also proposed a certain conjecture regarding to
the rigidity for long-time solutions of BMCFs in spheres
modeled on 
isoparametric hypersurfaces.
After the work of Liu-Terng, the analogous theory in
symmetric spaces of compact type has been still under development.
In this setting, a natural class of submanifolds generalizing isoparametric
hypersurfaces in spheres is that of equifocal submanifolds,
which was introduced by Terng-Thorbergsson (\cite{TeTh95}).
The second author (\cite{Koike11}) constructed explicit solutions of the MCF
starting from equifocal submanifolds in symmetric spaces of compact type,
and proved that the flow preserves equifocality and collapses to
a focal submanifold of the equifocal submanifold in finite time
in the case where the equifocal submanifold is not minimal.
Liu-Radeschi (\cite{LiuRa22}) proved that the BMCF
for an equifocal submanifold
has a long-time solution and converges to the unique minimal equifocal
submanifold as $t\to\infty$.
Subsequently, the authors (\cite{BK}) extended the results of Liu-Terng
for the explicit description of the squared norm of the
shape operator and the mean curvature along the long-time solution
of the BMCF to the case where the initial submanifold is an equifocal hypersurface
in a simply-connected symmetric space of compact type.
They also proposed the generalized Liu-Terng conjecture,
which predicts that an appropriate asymptotic behavior of a long-time solution of the BMCF
should imply that the initial hypersurface is equifocal.
Related developments have also been obtained in
the case where the ambient Riemannian manifold is a symmetric space of non-compact type.
In the forward direction, the second author \cite{Koike14} studied the MCF
for certain curvature-adapted isoparametric submanifolds
and their focal submanifolds in symmetric spaces of non-compact type.
Liu-Yang \cite{LiuYa26} investigated the MCF in both directions for isoparametric submanifolds
in hyperbolic spaces.

To approach the above conjectures, one needs more explicit estimates of the asymptotic behavior of long-time solutions of the BMCF for equifocal hypersurfaces.
When $N^{n+1}$ is a sphere, the estimates can be expressed in terms of the dimension $n$ and the number $g$ of distinct principal curvatures of the initial isoparametric hypersurface $M^{n}$
(see Remark \ref{rem:sphere_estimate}).
M\"unzner (\cite{M80, M81}) showed that $g$ must be $1,2,3,4$ or $6$.
Hence,
the dependence on $n$ and $g$ suggests a case-by-case approach to the study of general ancient solutions to BMCF in the sphere and, in particular, to the original Liu-Terng conjecture.
This approach should also be applicable when $N$ is a general simply-connected
irreducible symmetric space of compact type.
Motivated by this perspective,
the purpose of this paper
is to give explicit descriptions of
the results obtained in \cite{BK} 
for long-time solutions of BMCF
for principal orbits of cohomogeneity one Hermann actions.
This class provides fundamental examples of homogeneous equifocal hypersurfaces (\cite{HPTT}, \cite{TeTh95}).
For this purpose,
we briefly recall our results for general equifocal hypersurfaces.
Let $N$ be a simply-connected irreducible symmetric space of compact type
and denote by $\kappa$ the Einstein constant of $N$ (which is Einstein).
Let $M$ be a connected equifocal hypersurface in $N$
and $\Gamma^{F}(M)=\{\theta_{F},l/(2g_{F}),\{m^{F}_{1},m^{F}_{2}\};2g_{F}\}$
denote the tangential focal data of $M$.
See the statement of Theorem \ref{thm_TeTh95_thm1.6.b}
about this data.
We define constants $\delta_{F}$, $\epsilon_{F}$
by
\begin{equation}\label{eqn:deltaF_epsilonF}
\delta_{F}:=\dfrac{m_{1}^{F}-m_{2}^{F}}{m_{1}^{F}+m_{2}^{F}},
\quad
\epsilon_{F}:=\dfrac{1}{\pi^{2}}\left(\dfrac{l}{2g_{F}}\right)^{2}\dfrac{1}{m_{1}^{F}+m_{2}^{F}},
\end{equation}
and set $\bar{\theta}_{F}:=\pi (2g_{F}/l)\theta_{F}$.
Let $\{M_{t}\}_{t\in[0,\infty)}$ be the backward mean curvature flow starting from $M$.
Denote by $A(t)$ and $H(t)$ the shape tensor
and the mean curvature of $M_{t}$, respectively.
The authors (\cite{BK}) described the behavior of
the squared (Hilbert-Schmidt) norm
$\|A(t)\|^{2}$
and $H(t)$ along the BMCF
starting from $M$ in terms of
the data $m_{1}^{F}$, $m_{2}^{F}$,
$g_{F}$, $l$, $\delta_{F}$,
$\epsilon_{F}$, $\bar{\theta}_{F}$
and $\kappa$ when $M$ is not minimal (see Propositions 4.2, 4.3 and Theorems 4.5--4.7 in \cite{BK}).
These results will be stated in Section \ref{sec:MCF_equifocal}.

A \emph{Hermann action} is an isometric action of $K_{2}$ on
$G/K_{1}$, or of $K_{1}$ on $G/K_{2}$ for two symmetric spaces
$G/K_{1}$ and $G/K_{2}$ of compact type.
According to the classification due to Kollross (\cite[Theorem B]{Ko01}),
cohomogeneity one actions on
irreducible symmetric spaces of compact type are
orbit-equivalent to Hermann actions up to a few exceptions.
The triplet $(G,K_{1},K_{2})$ is called a \emph{compact symmetric triad}.
Based on the classification \cite[Table 2]{Ko01},
in Table~\ref{table:CST_rank1} below, we summarize compact symmetric triads
$(G,K_1,K_2)$ which give cohomogeneity one Hermann actions.
The principal curvatures of principal orbits of Hermann actions
are described by using root systems associated with $(G,K_{1},K_{2})$ (cf.~\cite{GoTh07}).
Since then, the geometry of these orbits has been further studied in terms of the notion
of symmetric triads (\cite{Ikawa11}, \cite{Oh23}).
Table~\ref{table:CST_rank1}
also lists the type of the symmetric triad
corresponding to each $(G,K_{1},K_{2})$.
By using this, we can and do determine the tangential focal data
to study MCFs and BMCFs starting from principal orbits of the cohomogeneity one Hermann action
(see Theorem \ref{thm:mainthm1} in Section \ref{sec:cohom1Hermann}).

\begin{table}[h!!]
\centering
\small
\renewcommand{\arraystretch}{1.1}
\caption{cohomogeneity one Hermann actions and their symmetric triads}\label{table:CST_rank1}
\begin{tabular}{cccc}
\toprule
No. &
\begingroup
\renewcommand{\arraystretch}{1.0}
\begin{tabular}{c}
Compact symmetric triad\\
$(G,K_{1},K_{2})$
\end{tabular}
\endgroup &
\begingroup
\renewcommand{\arraystretch}{1.0}
\begin{tabular}{c}
Type of\\
symmetric triad
\end{tabular}
\endgroup
& Remark \\
\toprule
1 &
$(SU(2),SO(2),S(U(1)\times U(1)))$ &
$(\mbox{IV-}A_{1})$ &
\\\midrule
2 &
$(SU(n),SO(n),S(U(1)\times U(n-1)))$ &
$(\mbox{II-}BC_{1})$ &
$n\geq 3$
\\\midrule
3 &
$(SU(n),S(U(1)\times U(n-1)),S(U(1)\times U(n-1)))$ &
$(\mbox{IV-}BC_{1})$ &
$n\geq2$
\\\midrule
4 &
$(SU(n),S(U(1)\times U(n-1)),S(U(c)\times U(n-c)))$ & 
$(\mbox{I-}BC_{1}\mbox{-}A_{1})$ &
\begingroup
\renewcommand{\arraystretch}{0.9}
\begin{tabular}{c}
$n\geq 4$,\\
$2\leq c \leq n-2$
\end{tabular}
\endgroup
\\\midrule
5 &
$(SU(4),Sp(2),S(U(2)\times U(2)))$ &
$(\mbox{III-}A_{1})$ &
\\\midrule
6 &
$(SU(2m),Sp(m),S(U(2)\times U(2m-2)))$ &
$(\mbox{III-}BC_{1})$ & 
$m \geq 3$
\\\midrule
7 &
$(SU(6),Sp(3),S(U(3)\times U(3)))$ &
$(\mbox{III-}BC_{1})$ &
\\\midrule
8 &
$(SU(2m),Sp(m),S(U(3)\times U(2m-3)))$ &
$(\mbox{V-}BC_{1}\mbox{-}A)$ &
$m\geq 4$
\\\midrule
9 &
$(SO(n),SO(n-1),SO(n-1))$ &
$(\mbox{IV-}A_{1})$ &
$n\geq 5$
\\\midrule
10 &
$(SO(n),SO(n-1),SO(c)\times SO(n-c))$ &
$(\mbox{III-}A_{1})$ &
\begingroup
\renewcommand{\arraystretch}{0.9}
\begin{tabular}{c}
$n\geq 5$,\\
$2\leq c \leq n-2$
\end{tabular}
\endgroup
\\
\bottomrule
\end{tabular}
\bigskip
\bigskip
\bigskip
\bigskip
\bigskip
\end{table}


\begin{table}[h!!]
\centering
\small
\renewcommand{\arraystretch}{1.1}
\caption*{\textsc{Table} \ref{table:CST_rank1}. (continued)}
\begin{tabular}{cccc}
\toprule
No. &
\begingroup
\renewcommand{\arraystretch}{1.0}
\begin{tabular}{c}
Compact symmetric triad\\
$(G,K_{1},K_{2})$
\end{tabular}
\endgroup &
\begingroup
\renewcommand{\arraystretch}{1.0}
\begin{tabular}{c}
Type of\\
symmetric triad
\end{tabular}
\endgroup
& Remark \\
\toprule
11 &
$(Sp(2),U(2),Sp(1)\times Sp(1))$ &
$(\mbox{III-}A_{1})$ &
\\\midrule
12 &
$(Sp(n),U(n),Sp(n-1)\times Sp(1))$ &
$(\mbox{III-}BC_{1})$ &
$n\geq 3$
\\\midrule
13 &
$(Sp(n),Sp(1)\times Sp(n-1),Sp(1)\times Sp(n-1))$ &
$(\mbox{IV-}BC_{1})$ &
$n\geq 2$
\\\midrule
14 &
$(Sp(n),Sp(1)\times Sp(n-1),Sp(c)\times Sp(n-c))$ &
$(\mbox{I-}BC_{1}\mbox{-}A_{1})$ &
\begingroup
\renewcommand{\arraystretch}{0.9}
\begin{tabular}{c}
$n\geq 4$,\\
$2\leq c\leq n-2$
\end{tabular}
\endgroup
\\\midrule
15 &
$(SO(2m),SO(2)\times SO(2m-2),U(m))$ &
$(\mbox{I-}BC_{1}\mbox{-}A_{1})$ &
$m\geq 3$
\\\midrule
16 &
$(SO(6),SO(3)\times SO(3),U(3))$ &
$(\mbox{II-}BC_{1})$ &
\\\midrule
17 &
$(SO(2m),SO(3)\times SO(2m-3),U(m))$ &
$(\mbox{V-}BC_{1}\mbox{-}D)$ &
$m\geq 4$
\\\midrule
18 &
$(E_{6},SU(6)\cdot SU(2),F_{4})$ &
$(\mbox{III-}BC_{1})$ &
\\\midrule
19 &
$(E_{6},Spin(10)\cdot U(1),F_{4})$ &
$(\mbox{III-}BC_{1})$ &
\\\midrule
20 &
$(F_{4},Sp(3) \cdot SU(2),Spin(9))$ &
$(\mbox{III-}BC_{1})$ &
\\\midrule
21 &
$(F_{4},Spin(9),Spin(9))$ &
$(\mbox{IV-}BC_{1})$ &
\\\bottomrule
\end{tabular}
\bigskip
\bigskip
\bigskip
\bigskip
\bigskip
\end{table}


The main purpose of this paper is to give of explicit descriptions of the behavior of the squared norm of the shape operator and the mean curvature along the BMCFs starting from principal orbits
for cohomogeneity one Hermann actions listed in Table~\ref{table:CST_rank1}.
Let $G$ be a compact connected simple Lie group with
the exponential map $\exp_{G}:\mathfrak{g}\to G$.
Let $(G,K_{1},K_{2})$ be a rank one compact symmetric triad.
Assume that 
$G/K_{i}$ ($i=1,2$) is simply-connected.
Let $\Sigma$ be the flat section of the $K_{j}$-action of $G/K_{i}$
through the origin $o$ of $G/K_{i}$ ($j=3-i$).
The orbit space of 
the $K_{j}$-action is identified with a closed subset of the tangent space $T_{o}\Sigma$,
which is given by
the closure of
a fundamental domain $C$ for
the symmetric triad $\mathcal{S}=(\Delta, \{\Delta_{\epsilon}\}_{\epsilon})$ associated with $(G,K_{1},K_{2})$
(see \eqref{eqn:dfn_fund_domain} for example).
Here, we note that 
$\Delta$ is a rank one root system;
hence,
if $\Delta$ is of type $A_{1}$, then it can be written as $\Delta=\{\pm\mu\}$,
while if $\Delta$ is of type $BC_{1}$, then it can be written as
$\Delta=\{\pm\mu,\pm2\mu\}$.
Any principal orbit $M$ of the $K_{j}$-action
intersects the subset $\pi_{i}(\exp_{G}C)$ in exactly one point,
which we write $\pi_{i}(\exp_{G}Z)$
for some $Z\in C$,
where $\pi_{i}:G\to G/K_{i}$
denotes the natural projection.
Then the MCF and the BMCF
starting from $M$
are reduced to a certain
ordinary differential
equation on $C$ (see \eqref{eqn:MCF_ODE_Hermann} for the detail).
Let $\{M_{t}\}_{t\in[0,\infty)}$
be the BMCF starting from $M$.
First, we describe explicitly
the exponential decay of the mean curvature $H(t)$,
and $\|A(t)\|^{2}$ as $t\to\infty$.
From \cite[Propositions 4.2, 4.3]{BK}
(see Fact \ref{thm:BMCF_evaluate} in Section \ref{sec:MCF_equifocal})
and Theorem \ref{thm:mainthm1},
we obtain the following result.

\begin{thm}\label{thm:mainthm2}
Under the above setting,
we have the followings:
\begin{enumerate}
\item[(i)] $\displaystyle\lim_{t\to\infty}H(t)^{2}\exp(\alpha\kappa t)
=\dfrac{\beta\kappa}{4}
(\delta_{F}-\cos(c\mu(Z)))^{2}$,
where $\alpha$ and $\beta$ are positive constant as in Table \ref{table:abg}
on the next page.
Here $c$ denotes the constant defined by
\begin{equation}\label{eqn:barthetaF_determ}
c:=\begin{cases}
1 & (\text{$\mathcal{S}=(\mbox{IV-}A_{1})$}),\\
2 & (\text{$\mathcal{S}\in\{(\mbox{III-}A_{1}), (\mbox{I-}BC_{1}\mbox{-}A_{1}),
(\mbox{IV-}BC_{1})
\}$}),\\
4 & (\text{otherwise}).
\end{cases}
\end{equation}
\item[(ii)] $\displaystyle\lim_{t\to\infty}\|A(t)\|^{2}=\gamma\kappa$,
where $\gamma$ is a non-negative constant as in Table \ref{table:abg}.
\end{enumerate}
\end{thm}

\begin{table}[h!!]
\centering
\footnotesize
\renewcommand{\arraystretch}{1.2}
\caption{Constants $\alpha$, $\beta$, $\gamma$ and $\psi_{F}$ in Theorems \ref{thm:mainthm2}--\ref{thm:main5}}\label{table:abg}
\scalebox{0.95}
{
\begin{tabular}{cccccc}
\toprule
No. & $\alpha$ & $\beta$ & $\gamma$ & $\psi_{F}$ & Remark\\
\toprule
1
& $2$
& $4$
& $0$
& $1$ 
&
\\\midrule
2
& $\dfrac{8(n-1)}{n}$
& $\dfrac{2(n-1)^4}{n(n-2)}$
& $\dfrac{3n-4}{n}$
& $\dfrac{4(n-2)}{(n-1)^{2}}$
& $n\geq 3$ 
\\\midrule
3
& $\dfrac{4(n-1)}{n}$
& $\dfrac{8(n-1)^4}{n(2n-3)}$
& $\dfrac{n-2}{n}$
& $\dfrac{2n-3}{(n-1)^{2}}$
& $n\geq 2$
\\\midrule
4
& $\dfrac{4(n-1)}{n}$
& $\dfrac{8(n-1)^4}{(2c-1)n(2n-1-2c)}$
& $\dfrac{n-2}{n}$
& $\dfrac{2n-2c-1}{(n-1)^{2}}$
& \begingroup
\renewcommand{\arraystretch}{0.9}
\begin{tabular}{c}
$n\geq 4$,\\
$2\leq c\leq n-2$
\end{tabular}
\endgroup
\\\midrule
5
& $4$
& $\dfrac{64}{3}$
& $1$
& $\dfrac{3}{4}$
&
\\\midrule
6
& $\dfrac{8(m-1)}{m}$
& $\dfrac{128(m-1)^4}{m(4m-5)}$
& $\dfrac{3m-4}{m}$
& $\dfrac{4m-5}{4(m-1)^{2}}$
& $m\geq 3$
\\\midrule
7
& $\dfrac{16}{3}$
& $\dfrac{2048}{45}$
& $\dfrac{5}{3}$
& $\dfrac{5}{16}$
&
\\\midrule
8
& $\dfrac{8(m-1)}{m}$
& $\dfrac{128(m-1)^4}{5m(4m-9)}$
& $\dfrac{3m-4}{m}$
& $\dfrac{4m-9}{4(m-1)^{2}}$
& $m\geq 4$
\\\midrule
9
& $2$
& $4(n-2)$
& $0$
& $\dfrac{1}{n-2}$
& $n\geq5$
\\\midrule
10
& $4$
& $\dfrac{(n-2)^3}{(c-1)(n-1-c)}$
& $1$
& $\dfrac{4n-4c-1}{(n-2)^{2}}$
& \begingroup
\renewcommand{\arraystretch}{0.9}
\begin{tabular}{c}
$n\geq 5$,\\
$2\leq c\leq n-2$
\end{tabular}
\endgroup
\\\midrule
11
& $4$
& $\dfrac{27}{2}$
& $1$
& $\dfrac{8}{9}$
&
\\\midrule
12
& $\dfrac{4(2n-1)}{n+1}$
& $\dfrac{(2n-1)^4}{2(n+1)(2n-3)}$
& $\dfrac{3(n-1)}{n+1}$
& $\dfrac{4(2n-3)}{(2n-1)^{2}}$
& $n\geq 3$
\\\midrule
13
& $\dfrac{2(2n-1)}{n+1}$
& $\dfrac{4(2n-1)^4}{3(n+1)(4n-5)}$
& $\dfrac{n-2}{n+1}$
& $\dfrac{4n-5}{(2n-1)^{2}}$
& $n\geq 2$
\\\midrule
14
& $\dfrac{2(2n-1)}{n+1}$
& $\dfrac{4(2n-1)^4}
{(4c-1)(n+1)(4n-1-4c)}$
& $\dfrac{n-2}{n+1}$
& $\dfrac{4n-4c-1}{(2n-1)^{2}}$
& \begingroup
\renewcommand{\arraystretch}{0.9}
\begin{tabular}{c}
$n\geq 4$,\\
$2\leq c\leq n-2$
\end{tabular}
\endgroup
\\\midrule
15
& $\dfrac{2(2m-3)}{m-1}$
& $\dfrac{4(2m-3)^2}{m-1}$
& $\dfrac{m-2}{m-1}$
& $\dfrac{1}{2m-3}$
& $m\geq 3$
\\
\bottomrule
\end{tabular}}
\bigskip
\end{table}

\begin{table}[h!!]
\centering
\footnotesize
\renewcommand{\arraystretch}{1.1}
\setlength{\tabcolsep}{12pt}
\caption*{\textsc{Table} \ref{table:abg}. (continued)}
{
\begin{tabular}{cccccc}
\toprule
No. & $\alpha$ & $\beta$ & $\gamma$ & $\psi_{F}$ & Remark\\
\toprule
16
& $6$
& $\dfrac{81}{4}$
& $2$
& $\dfrac{8}{9}$
\\\midrule
17
& $\dfrac{4(2m-3)}{m-1}$
& $\dfrac{(2m-3)^4}{2(m-1)(2m-5)}$
& $\dfrac{3m-5}{m-1}$
& $\dfrac{4(2m-5)}{(2m-3)^{2}}$
& $m\geq 4$
\\\midrule
18
& $\dfrac{16}{3}$
& $\dfrac{16384}{165}$
& $\dfrac{5}{3}$
& $\dfrac{11}{64}$
&
\\\midrule
19
& $\dfrac{16}{3}$
& $\dfrac{16384}{45}$
& $\dfrac{5}{3}$
& $\dfrac{15}{64}$
&
\\\midrule
20
& $\dfrac{44}{9}$
& $\dfrac{14641}{252}$
& $\dfrac{13}{9}$
& $\dfrac{28}{121}$
&
\\\midrule
21
& $\dfrac{22}{9}$
& $\dfrac{58564}{945}$
& $\dfrac{2}{9}$
& $\dfrac{15}{121}$
&
\\
\bottomrule
\end{tabular}
}
\end{table}

There exists a unique minimal principal $K_{j}$-orbit
in $G/K_{i}$, which we write $M_{\min}$.
Then,
by \cite[Corollary 3.2]{BK},
$M_{\min}$
is given by the $K_{j}$-orbit $K_{j}\cdot\pi_{i}(\exp_{G}Z_{\min})$
for $Z_{\min}\in C$
which is characterized by
$\delta_{F}-\cos(c\mu(Z_{\min}))=0$.
It is known that $\{M_{t}\}_{t\in[0,\infty)}$ converges to $M_{\min}$
as $t\to \infty$ (cf.~\cite{LiuRa22}).
From \cite[Theorem 4.5]{BK} (see Fact \ref{thm:main3} in Section \ref{sec:MCF_equifocal})
and
Theorem \ref{thm:mainthm1},
we obtain the following result.

\begin{thm}\label{thm:main3-1}
Assume that $M_{0}=M$ is not minimal.
Then we obtain the followings:
\begin{enumerate}
\item[(1)] In the case when $(G,K_{1},K_{2})$ is one of
$(SU(2),S(U(1)\times U(1)), S(U(1)\times U(1)))$,
$(SO(n),SO(n-1), SO(n-1))$ with $n\geq 5$,
or $(Sp(2),Sp(1)\times Sp(1), Sp(1)\times Sp(1))$,
we have
\begin{equation}
\dfrac{\|A(t)\|^{2}}{H(t)^{2}}
=\begin{cases}
1 & (SU(2),SO(2), S(U(1)\times U(1))),\\
1/(n-2) & (SO(n),SO(n-1), SO(n-1)),\\
1/3 & (Sp(2),Sp(1)\times Sp(1), Sp(1)\times Sp(1))
\end{cases}
\end{equation}
\item[(2)] Otherwise, the following inequality holds:
\begin{equation}
\dfrac{\|A(t)\|^{2}}{H(t)^{2}}<\dfrac{\psi_{F}}{(\delta_{F}-\cos(c\mu(Z)))}\cdot\exp(\alpha \kappa t),
\end{equation}
where $\psi_{F}$ is a positive constant as in Table \ref{table:abg}.
\end{enumerate}
\end{thm}

\begin{rem}
In each case of Theorem \ref{thm:main3-1}, the symmetric space $G/K_{i}$ is a sphere,
namely $S^{2}$, $S^{n-1}$ and $S^{4}$, respectively,
and the $K_{j}$-action on $G/K_{i}$ is orbit equivalent to the isotropy action
of $SO(2)$ on $S^{2}$, of $SO(n-1)$ on $S^{n-1}$ and of $SO(4)$ on $S^{4}$,
respectively. 
\end{rem}

From \cite[Theorem 4.6]{BK}
(see Fact \ref{thm:BMCF_evaluate2} in Section \ref{sec:MCF_equifocal})
and
Theorem \ref{thm:mainthm1},
we obtain the following result.

\begin{thm}\label{thm:main4}
Assume that
$M_{0}=M$ is not minimal
and that $(G,K_{1},K_{2})$
is not one of the exceptional cases in Theorem \ref{thm:main3-1}, (1).
Then 
there exist $t_{0}>0$ and $c_{1},c_{2}>0$ satisfying the following inequalities:
\begin{equation}\label{eqn:conj1}
c_{2}\cdot\exp(\alpha\kappa t)
\leq 
\dfrac{\|A(t)\|^{2}}{H(t)^{2}}
\leq c_{1}\cdot\exp(\alpha\kappa t)
\quad
(t>t_{0}).
\end{equation}
\end{thm}

Let us consider an estimate for the traceless part 
$\mathring{A}(t)$ of $A(t)$
along the backward mean curvature flow $\{M_{t}\}_{t\in[0,\infty)}$.
Let $\phi(t)$
denote the squared norm of $\mathring{A}(t)$,
that is,
\begin{equation}
\phi(t):=\|\mathring{A}(t)\|^{2},
\quad
\mathring{A}(t):=A(t)-\dfrac{H(t)}{\dim M_{t}}\mathrm{id}.
\end{equation}
We set $\theta_{\min}=l/(2g_{F})-t_{\min}$, where $Z_{\min}=t_{\min}\xi$
for some $t_{\min}\in (0,l/(2g_{F}))$.
From \cite[Theorem 4.7]{BK}
(see Fact \ref{thm:BMCF_evaluate3} in Section \ref{sec:MCF_equifocal})
and
Theorem \ref{thm:mainthm1},
we obtain the following result.

\begin{thm}\label{thm:main5}
For any $0<h<1$, there exists $c_{0}>0$ such that
one of the following assertions holds:
\begin{enumerate}
\item[(1)] The case of $\delta_{F}=0$:
If $\theta_{F}\in (\theta_{\min}-c_{0},\theta_{\min}+c_{0})$, we have
\begin{equation}\label{eqn:cor_phi1}
\kappa\gamma
\leq
\phi(t)
\leq
\kappa\gamma+h\quad
(t>0).
\end{equation}
\item[(2)] The case of $\delta_{F}>0$: If $\theta_{F}\in (\theta_{\min}-c_{0},\theta_{\min})$,
we have
\begin{equation}\label{eqn:cor_phi2}
\kappa\gamma-h
\leq
\phi(t)
\leq
\kappa\gamma
\quad(t>0).
\end{equation}
If $\theta_{F}\in (\theta_{\min},\theta_{\min}+c_{0})$, we have
\begin{equation}\label{eqn:cor_phi3}
\kappa\gamma\leq
\phi(t)
\leq
\kappa\gamma+h
\quad
(t>0).
\end{equation}
\item[(3)] The case of $\delta_{F}<0$:
If $\theta_{F}\in (\theta_{\min}-c_{0},\theta_{\min})$,
we have
\begin{equation}\label{eqn:cor_phi4}
\kappa\gamma
\leq
\phi(t)
\leq
\kappa\gamma+h
\quad
(t>0).
\end{equation}
If $\theta_{F}\in (\theta_{\min},\theta_{\min}+c_{0})$,
we have
\begin{equation}\label{eqn:cor_phi5}
\kappa\gamma-h
\leq
\phi(t)
\leq
\kappa\gamma
\quad(t>0).
\end{equation}
\end{enumerate}
\end{thm}


The paper is organized as follows.
In Section \ref{sec:pre}, we recall the notion of equifocal submanifolds,
and explain the data associated with
equifocal hypersurfaces.
In Section \ref{sec:MCF_equifocal}, we recall our results
on BMCFs for equifocal hypersurfaces.
In Section \ref{sec:cohom1Hermann}, we determine the tangential focal data for principal
orbits of cohomogeneity one Hermann actions.
In Section \ref{sec:MCF_Hermann}
we review the MCF and the BMCF
for principal orbits of cohomogeneity one Hermann actions
and prove Theorems \ref{thm:mainthm2}--\ref{thm:main5}.

\section{Equifocal hypersurfaces}\label{sec:pre}

\subsection{Basics of equifocal submanifolds}\label{sec:bes}

Let $N$ be a complete Riemannian manifold,
$R$ the Riemannian curvature tensor field of $N$
and $\exp_{x}$ the exponential map of $N$ at $x$.
Let $M$ be an immersed submanifold of $N$
and $\pi^{\perp}:T^{\perp}M\to M$ denote the normal bundle of $M$.
The normal exponential map of $M$ is defined by
\begin{equation}
 \exp^{\perp}(\bs{\xi}):=\exp_{x}(\bs{\xi})
 \quad
 (x\in M,\,
 \bs{\xi}\in T^{\perp}_{x}M).
\end{equation}
A normal vector $\bs{\xi}\in T^{\perp}_{x}M$
($x\in M$)
is said to be singular for $\exp^{\perp}$
if $d\pi^{\perp}_{\bs{\xi}}(\operatorname{Ker}d\exp^{\perp}_{\bs{\xi}})\neq\{0\}$ holds,
in which case $\bs{\xi}$
is called a \emph{focal normal} at $x$
and the point $\exp^{\perp}(\bs{\xi})\in N$
is called a \emph{focal point} of $M$ at $x$.
For any focal normal $\bs{\xi}$,
the integer $m^{F}_{\bs{\xi}}:=\dim_{\mathbb{R}}d\pi^{\perp}_{\bs{\xi}}(\operatorname{Ker}d\exp^{\perp}_{\bs{\xi}})$
is called its \emph{multiplicity},
and its length $\|\bs{\xi}\|$ is called its \emph{focal radius}.
The set consisting of the focal normals at $x$
and their
multiplicities is called the \emph{tangential focal data} of $M$ at $x$, which we write
$\Gamma^{F}(M,x)$. The set
$ \Gamma^{F}(M)=\bigcup_{x\in M}\Gamma^{F}(M,x)$
is called the \emph{tangential focal data} of $M$.

\begin{dfn}[{\cite{TeTh95}}]\label{dfn:equifocal_submaifold}
Let $N$ be a symmetric space of compact type. A compact immersed
submanifold $M$ of $N$ is called an \emph{equifocal submanifold} if it
satisfies the following three conditions:
\begin{enumerate}
\item[(1)] $T^{\perp}M$ is globally flat,
that is,
$T^{\perp}M$ is flat with respect to $\nabla^{\perp}$ and has trivial holonomy group.
\item[(2)] $T^\perp M$ is abelian, that is,
for each $x\in M$, the set $\exp_{x}(T^{\perp}_{x}M)$ is contained in
a flat totally geodesic submanifold in $N$.
\item[(3)] The tangential focal data $\Gamma^{F}(M)$ is invariant under parallel
translation with respect to the normal connection $\nabla^{\perp}$.
\end{enumerate}
Here,
$\exp_{x}(T^{\perp}_{x}M)$
is called a \emph{flat section} of $M$.
\end{dfn}

The notion of an equifocal submanifold is modeled on two fundamental classes of submanifolds:
isoparametric hypersurfaces in spheres and principal orbits of hyperpolar actions
on symmetric spaces of compact type.

Let $M$ be an equifocal submanifold in $N$, and let $\bs{\xi}$ be
a parallel normal vector field on $M$.
For a real number $t$,
the map $\eta_{t\bs{\xi}}:M\to N$
defined by $\eta_{t\bs{\xi}}(x)=\exp^{\perp}(t\bs{\xi}_{x})$
for $x\in M$ is called the \emph{endpoint map} of $M$ for $t\bs{\xi}$.
According to (3) of Definition \ref{dfn:equifocal_submaifold},
$\eta_{t\bs{\xi}}$ is of constant rank.
The image $M_{t\bs{\xi}}:=\eta_{t\bs{\xi}}(M)$
of $\eta_{t\bs{\xi}}$
is called a \emph{parallel submanifold} of $M$
if $\eta_{t\bs{\xi}}$ is of full rank;
otherwise, $M_{t\bs{\xi}}$ is called a \emph{focal submanifold} of $M$.
Any parallel submanifold of $M$
is diffeomorphic to $M$ via $\eta_{t\bs{\xi}}$
and an equifocal submanifold in $N$.

We explain the relation between focal radii of $M$ and Jacobi fields along
normal geodesics.
Let $\bs{\xi}$ be a parallel normal vector field on $M$.
For a normal vector $\bs{v}\in T^{\perp}_{x}M$ ($x\in M$),
we denote by $A_{\bs{v}}$ the shape operator of $M$ with respect to $\bs{v}$.
Let $\gamma_{\bs{\xi}_{x}}$ denote the normal geodesic of $M$
satisfying $\gamma_{\bs{\xi}_{x}}(0)=x$
and
$\gamma_{\bs{\xi}_{x}}'(0)=\bs{\xi}_{x}$.
For $X\in T_{x}M$, let $J_X$ denote
the Jacobi field along $\gamma_{\bs{\xi}_x}$ satisfying
$J_{X}(0)=X$ and $J_{X}'(0)=-A_{\bs{\xi}_{x}}X$.
The Jacobi field $J_{X}$ is a solution of the linear ordinary differential
equation (abbreviated as ODE)
\begin{equation}\label{eqn:Jacobi_ODE}
J_{X}''(t)+R_{\bs{\xi}(t)}(J_{X}(t))=0,
\end{equation}
where $\bs{\xi}(t):=\gamma_{\bs{\xi}_{x}}'(t)$
and $R_{\bs{\xi}(t)}$
denotes the linear operator on $T_{\gamma_{\bs{\xi}_{x}}(t)}N$ defined by
$R_{\bs{\xi}(t)}(\bullet):=R(\bullet,\bs{\xi}(t))\bs{\xi}(t)$.
Then we have $J_{X}(t)=d(\eta_{t\bs{\xi}})_{x}X$.
In addition,
it is known that
$d\pi^{\perp}_{t\bs{\xi}_{x}}(\operatorname{Ker}d\exp^{\perp}_{t\bs{\xi}_{x}})=\operatorname{Ker}(d\eta_{t\bs{\xi}})_{x}$ holds.
Therefore,
for $t_{*}\in\mathbb{R}$,
$t_{*}\cdot\bs{\xi}_{x}$ is a focal normal if and only if
there exists a non-zero vector $X\in T_{x}M$ such that $J_{X}(t_{*})=0$.
The Jacobi field $J_{X}(t)$ along $\gamma_{\bs{\xi}_{x}}$ satisfying
$J_{X}(0)=X$ and
$J_{X}'(0)=-A_{\bs{\xi}_{x}}X$ is given by
\begin{align}
J_{X}(t)
&=
P_{\gamma_{\bs{\xi}_{x}}}|_{[0,t]}(
D^{co}_{t\bs{\xi}_{x}}J_{X}(0)+D^{si}_{t\bs{\xi}_{x}}J_{X}'(0))
\\
&=
P_{\gamma_{\bs{\xi}_{x}}}|_{[0,t]}
(D^{co}_{t\bs{\xi}_{x}}-D^{si}_{t\bs{\xi}_x}\circ A_{\bs{\xi}_{x}})X.
\label{eqn:Jacob_sol_general}
\end{align}
Here, $P_{\gamma_{\bs{\xi}_{x}}}|_{[0,t]}$ denotes the parallel translation
along $\gamma_{\bs{\xi}_{x}}$
and
the linear operators
$D^{co}_{t\bs{\xi}_{x}}$
and $D^{si}_{t\bs{\xi}_{x}}$
on $T_{x}M$ are defined as follows:
\begin{equation}
D^{co}_{t\bs{\xi}_{x}}(\bullet)
:=\sum_{k=0}^{\infty}
\frac{(-1)^{k} t^{2k}}{(2k)!}
R_{\bs{\xi}_{x}}^{k}(\bullet),
\qquad
D^{si}_{t\bs{\xi}_{x}}(\bullet)
:=
\sum_{k=0}^{\infty}
\frac{(-1)^{k} t^{2k+1}}{(2k+1)!}
R_{\bs{\xi}_{x}}^{k}(\bullet).
\end{equation}
Hence, for $t_{*}\in\mathbb{R}$, the normal vector
$t_{*}\bs{\xi}_x$ is a focal normal if and only if 
the following function
$F_{\bs{\xi}_{x}}$ vanishes at $t_{*}$:
\begin{equation}\label{eqn:Fxi_dfn}
F_{\bs{\xi}_{x}}(t)
:=
\det(D^{co}_{t\bs{\xi}_x}-D^{si}_{t\bs{\xi}_{x}}\circ A_{\bs{\xi}_{x}}).
\end{equation}

\subsection{Tangential focal data of equifocal hypersurfaces}

We recall some results on equifocal hypersurfaces based on
\cite{TeTh95}.
Let $N$ be a symmetric space of compact
type.
Throughout of this paper,
we always assume that $N$ is simply-connected.
Let $M$ be an equifocal hypersurface in $N$,
and $\bs{\xi}$ be a unit normal vector field on $M$.
Let $x\in M$
and $\gamma_{\bs{\xi}_{x}}$ 
be as in Subsection \ref{sec:bes}.
By \cite[Theorem 1.6, (a)]{TeTh95},
the image $T_{x}:=\operatorname{Im}(\gamma_{\bs{\xi}_x})$
is a circle, which is called a \emph{normal circle} of $M$ at $x$.
Furthermore, the length of $T_{x}$ is independent of the choice of $x$.
We denote this common value by $l$.
It follows that the end-point map $\eta_{l\cdot \bs{\xi}}$
can be regarded as the identity map on $M$:
\begin{equation}\label{eqn:etalxi_id}
\eta_{l\cdot\bs{\xi}}=\operatorname{id}_{M}.
\end{equation}
The following theorem is also due to Terng-Thorbergsson \cite{TeTh95}.

\begin{thm}[{\cite[Theorem 1.6, (b)]{TeTh95}}]
\label{thm_TeTh95_thm1.6.b}
There exist a real number $\theta_{F}$ and integers
$2g_F$, $m^{F}_{1}$, and $m^{F}_{2}$ satisfying the following conditions:
\begin{enumerate}
\item[(1)] $0<\theta_{F}<\dfrac{l}{2g_F}$.
\item[(2)] The set $\mathcal{F}_{x}$ of focal normals of $M$ at $x$ is
written as
\begin{equation}
\mathcal{F}_{x}
=
\left\{t_{j,k}\cdot\bs{\xi}_{x}
\,\middle|\,
j\in\{1,\dotsc,2g_{F}\},\, k\in\mathbb{Z}\right\}.
\end{equation}
Here $t_{j,k}\in\mathbb{R}$ is defined by
\begin{equation}\label{eqn:tjk_dfn}
t_{j,k}:=\theta_{F}+(j-1)\dfrac{l}{2g_{F}}+l\cdot k.
\end{equation}
Furthermore, for each $j\in\{1,\dotsc,2g_{F}\}$ and $k\in\mathbb{Z}$,
the multiplicity of $t_{j,k}\cdot\bs{\xi}_{x}$ is given by
\begin{equation}\label{eqn:multF_inv}
m^{F}_{t_{j,k}\cdot\bs{\xi}_{x}}=m^{F}_{j\,(\operatorname{mod} 2)},
\end{equation}
where $m^{F}_{j\,(\operatorname{mod} 2)}$ means $m^{F}_{1}$ for odd $j$ and
$m^{F}_{2}$ for even $j$.
\item[(3)] The group generated by the affine reflections on
$T^{\perp}_{x}M$ with respect to the focal normals of $M$ at $x$ is
isomorphic to the infinite dihedral group.
\end{enumerate}
\end{thm}

The group stated in Theorem \ref{thm_TeTh95_thm1.6.b}, (3), is called
the \emph{affine Coxeter group} of the equifocal hypersurface $M$ at
$x$, which we write $W_{x}$.
By Theorem \ref{thm_TeTh95_thm1.6.b}, (2),
the group $W_x$ preserves $\mathcal{F}_{x}$ invariantly.
In this case, the orbit space $\mathcal{F}_{x}/W_{x}$
is given by $\{t_{1,0}\cdot\bs{\xi}_{x}(=\theta_{F}\cdot\bs{\xi}_{x}), t_{2,0}\cdot\bs{\xi}_{x}\}$.
In other words, \eqref{eqn:multF_inv} means that the multiplicity
$m^F_{\bs{v}}$ $(\bs{v}\in\mathcal{F}_{x})$ is invariant under the action
of $W_{x}$.
Based on this observation, we rewrite the tangential focal data
$\Gamma^{F}(M,x)$ of the equifocal hypersurface $M$ at $x$ by
\begin{equation}
 \Gamma^{F}(M,x)
 =
 \left\{
\theta_{F},\dfrac{l}{2g_F},\{m_{1}^{F},m_{2}^{F}\};2g_F
 \right\}.
\end{equation}
On the other hand, by Theorem \ref{thm_TeTh95_thm1.6.b},
the set
$\{\gamma_{\bs{\xi}_x}(t_{j,0})\mid j\in\{1,\dotsc,2g_F\}\}$
coincides with the set of focal points on the normal circle $T_{x}$.
These focal points are arranged in $T_{x}$ at equal intervals of $l/(2g_{F})$.
In particular, the two points $\gamma_{\bs{\xi}_{x}}(t_{j,0})$ and
$\gamma_{\bs{\xi}_x}(t_{j+g_{F},0})$ are in antipodal position on $T_{x}$.
Here, by Theorem~\ref{thm_TeTh95_thm1.6.b}, (2), we note that the following equality holds:
\begin{equation}\label{eqn:F_0set_focal_radii}
F_{\bs{\xi}_{x}}^{-1}(0)
=
\{t_{j,k}\mid j\in\{1,\dotsc,2g_{F}\},\, k\in\mathbb{Z}\}.
\end{equation}

The constancy of $l$ implies the periodicity of Jacobi fields along $\gamma_{\bs{\xi}_{x}}$.
Namely, the following proposition holds.

\begin{pro}\label{pro:JX_period_l}
Let $X\in T_{x}M$
and $J_X$ denote
the Jacobi field along $\gamma_{\bs{\xi}_x}$ satisfying
$J_{X}(0)=X$ and $J_{X}'(0)=-A_{\bs{\xi}_{x}}X$.
Then $J_{X}(t+l)=J_{X}(t)$ holds.
\end{pro}

\begin{proof}
We set $K_{X}(t):=J_{X}(t+l)$. Then
$K_{X}(t)$ satisfies
\begin{equation}
K_{X}''(t)
+R_{\gamma_{\bs{\xi}_{x}}'(t)}(K_{X}(t))
=
J_{X}''(t+l)
+R_{\gamma_{\bs{\xi}_{x}}'(t+l)}(J_{X}(t+l))
=0.
\end{equation}
We also get the following equalities:
\begin{align}
K_{X}(0)&=J_{X}(l)=d(\eta_{l\cdot\bs{\xi}})_{x}(X)\overset{\eqref{eqn:etalxi_id}}{=}X=J_{X}(0),\\
K_{X}'(0)
&=-A_{\gamma_{\bs{\xi}_{x}}'(l)}J_{X}(l)
=-A_{\bs{\xi}_{x}}X
=J_{X}'(0).
\end{align}
Therefore, by the uniqueness of solutions to linear ODE,
we obtain $K_{X}(t)=J_{X}(t)$.
\end{proof}

\subsection{Eigenvalues of the normal Jacobi operator of equifocal hypersurfaces}\label{sec:eigen_equihyper}
Let $M$ be an equifocal hypersurface in $N=G/K$,
and $\bs{\xi}$ be a unit normal vector field on $M$.
Without loss of generalities, we may assume that
the origin $o:=eK$ is contained in $M$.
Let $\mathfrak{g}:=\operatorname{Lie}(G)$
denote the Lie algebra of $G$.
We write $\mathfrak{g}=\mathfrak{k}\oplus\mathfrak{p}$
as the Cartan decomposition of $\mathfrak{g}$
associated with $G/K$.
Here, $\mathfrak{k}=\operatorname{Lie}(K)$
holds, and $\mathfrak{p}$ is identified with the tangent space of $G/K$
at $o$ via the natural projection $\pi:G\to G/K$.
We regard the tangent space $T_{o}M$
and the normal space $T^{\perp}_{o}M$
as subspaces of $\mathfrak{p}$.
We set $\mathfrak{b}:=T^{\perp}_{o}M=\mathbb{R}\bs{\xi}_{o}$.
For simplicity,
we write $\xi$ instead of $\bs{\xi}_{o}$.
For $\mu\in\mathfrak{b}^{*}$,
we define the subspace $\mathfrak{p}(\mathfrak{b},\mu)$ of $\mathfrak{p}$
as follows:
\begin{equation}\label{eqn:pbmu_dfn}
\mathfrak{p}(\mathfrak{b},\mu):=\{X\in\mathfrak{p}\mid
\operatorname{ad}^{2}_{\xi}X=-\mu(\xi)^{2}X\}
\end{equation}
Then we have $\mathfrak{p}(\mathfrak{b},\mu)=\mathfrak{p}(\mathfrak{b},-\mu)$.
We put
$\Delta:=\{\mu\in\mathfrak{b}^{*}\mid \mathfrak{p}(\mathfrak{b},\mu)\neq\{0\}\}-\{0\}$
and $\Lambda^{R}:=\{\mu\in\Delta\mid \mu(\xi)>0\}$.
We have
\begin{equation}\label{eqn:decomp_p_delta}
\mathfrak{p}
=
\mathfrak{p}(\mathfrak{b},0)
\oplus\bigoplus_{\mu\in\Lambda^{R}}\mathfrak{p}(\mathfrak{b},\mu).
\end{equation}
Since $\mathfrak{b}$
is contained in $\mathfrak{p}(\mathfrak{b},0)$,
we obtain
\begin{equation}\label{eqn:decomp_p}
T_{o}M=\mathfrak{p}\ominus\mathfrak{b}
=
(\mathfrak{p}(\mathfrak{b},0)\ominus\mathfrak{b})
\oplus\bigoplus_{\mu\in\Lambda^{R}}
\mathfrak{p}(\mathfrak{b},\mu).
\end{equation}
Let $R_{\xi}$
denote the linear operator on $T_{o}N$ defined by
$R_{\xi}(\bullet):=R(\bullet,\xi)\xi$.
Then we have $R_{\xi}(T_{o}M)\subset T_{o}M$.
We use the same symbol $R_{\xi}$
as the restriction of $R_{\xi}$
to $T_{o}M$,
which gives a self-adjoint operator on $T_{o}M$.
Since $R_{\xi}=-\operatorname{ad}^{2}_{\xi}|_{T_{o}M}$ holds,
\eqref{eqn:decomp_p}
gives the eigenspace decomposition of $T_{o}M$
for $R_{\xi}$
and $\Lambda^{R}$ coincides with the set of positive eigenvalues of $R_{\xi}$.
The normal circle $T_{o}$
of $M$ at $o$ satisfies the
following periodicity condition.

\begin{pro}\label{pro:normal_circle_period}
Let $l$ be the length of $T_{o}$.
For each $\mu\in\Lambda^{R}$, we have
$\mu(l\cdot\xi)\in \pi\mathbb{Z}$.
In particular, there exists a unique positive integer $l_{\mu}$ such that
\begin{equation}\label{eqn:Tx_period}
 \mu(l\cdot\xi)=\pi\cdot l_{\mu}.
\end{equation}
\end{pro}

We make some preparation for the proof of this proposition.
For $\mu\in\mathfrak{b}^{*}$,
we define the subspace 
$\mathfrak{k}(\mathfrak{b},\mu)$
of $\mathfrak{k}$ as follows:
\begin{equation}
\mathfrak{k}(\mathfrak{b},\mu)
:=\{X\in\mathfrak{k}\mid
\mathrm{ad}_{\xi}^{2}(X)=-\mu(\xi)^{2}X\}.
\end{equation}
We get $\dim\mathfrak{k}(\mathfrak{b},\mu)=\dim\mathfrak{p}(\mathfrak{b},\mu)=:\tilde{m}_{\mu}$,
which is called the multiplicity of $\mu$.

\begin{lem}\label{lem:pbmu_basis}
Under the above setting,
for each $\mu\in\Lambda^{R}$,
there exist orthonormal bases
$\{S_{i}\}_{i=1}^{\tilde{m}_{\mu}}$
of $\mathfrak{p}(\mathfrak{b},\mu)$
and 
$\{T_{i}\}_{i=1}^{\tilde{m}_{\mu}}$
of $\mathfrak{k}(\mathfrak{b},\mu)$
satisfying the following conditions:
for any $v\in\mathfrak{b}$,
\begin{equation}
\begin{cases}
[v,S_{i}]=\mu(v)T_{i},\quad
[v,T_{i}]=-\mu(v)S_{i},\quad
[S_{i},T_{i}]=\mu^{\#},\\
\mathrm{Ad}(\exp_{G} v)S_{i}
=\cos\mu(v)S_{i}+\sin\mu(v)T_{i},\\
\mathrm{Ad}(\exp_{G}v)T_{i}
=-\sin\mu(v)S_{i}+\cos\mu(v)T_{i},
\end{cases}
\end{equation}
where $\mu^{\#}$ is the vector of $\mathfrak{b}$
defined by $\mu(v)=\INN{\mu^{\#}}{v}$
for all $v\in\mathfrak{b}$.
\end{lem}

\begin{proof}
As in \cite{Ta}, we omit the details.
\end{proof}

\begin{proof}[Proof of Proposition \ref{pro:normal_circle_period}]
We keep the above notation and assumptions.
Let $\exp_{G}:\mathfrak{g}\to G$ denote the exponential map of $G$.
We have
$\gamma_{\xi}(t)
=\pi(\exp_{G}(t\cdot\xi))=\exp_{G}(t\cdot\xi)o$.
By substituting $t=l$ into this,
we get
$o=\exp_{G}(l\cdot\xi)o$.
Hence $\exp_{G}(l\cdot\xi)\in K$ holds.
This yields
\begin{equation}\label{eqn:Adexpls_p}
\mathrm{Ad}(\exp_{G}(l\cdot\xi))S_{i}\in\mathfrak{p}.
\end{equation}
From Lemma \ref{lem:pbmu_basis},
we get
\begin{equation}\label{eqn:Adexpls_p2}
\mathrm{Ad}(\exp_{G}(l\cdot\xi))S_{i}
=\cos\mu(l\cdot\xi)S_{i}+\sin\mu(l\cdot\xi)T_{i}.
\end{equation}
By comparing \eqref{eqn:Adexpls_p} with
\eqref{eqn:Adexpls_p2}
we obtain
$\sin\mu(l\cdot\xi)=0$,
and hence $\mu(l\cdot\xi)\in\pi\mathbb{Z}$ holds.
Thus, we have completed the proof.
\end{proof}

\begin{rem}
The element $\mu\in\Lambda^{R}$ and $l\cdot\xi$
in Proposition \ref{pro:normal_circle_period}
do not depend on the scale of $N$.
Therefore, the integer $l_{\mu}$ is also determined independently of the scale of $N$.
\end{rem}

\section{Backward mean curvature flows for equifocal hypersurfaces}\label{sec:MCF_equifocal}

Let $N$ be a simply-connected, irreducible symmetric space of compact type
and $\kappa$ denote the Einstein constant of $N$.
In this section, we give a review on
the descriptions of
the mean curvature and
the squared norm of the shape operator along the MCF in the backward direction starting from
equifocal hypersurfaces in $N$,
and to give the asymptotic analysis of the long-time solutions based on \cite{BK}.
Let $M$ be a connected equifocal hypersurface in $N$
and $\Gamma^{F}(M)=\{\theta_{F},l/(2g_{F}),\{m_{1}^{F},m_{2}^{F}\};2g_{F}\}$
denote the tangential focal data of $M$.
Due to the second author
(\cite{Koike11})
and
Liu-Radeschi
the BMCF starting from an equifocal hypersurface $M$
exist in long time.
Let $\{M_{t}\}_{t\in[0,\infty)}$
be the BMCF starting from $M$.
By \cite{Koike11},
$\{M_{t}\}_{t\in[0,\infty)}$ is given by the family of parallel hypersurfaces of $M$.
Let $A(t)$ and $H(t)$ denote
the shape operator and the mean curvature of $M_{t}$, respectively.
There exists a unique minimal equifocal hypersurface among
the parallel equifocal hypersurfaces of $M$, which we write $M_{\min}$.
Then it follows from \cite[Corollary 3.2]{BK}
that $M_{\min}$ is characterized as the equifocal hypersurface satisfying
$\delta_{F} + \cos\bar{\theta}_{F} = 0$,
where $\bar{\theta}_{F}$ is defined by $\bar{\theta}_{F}:=\pi (2g_{F}/l)\theta_{F}$.
Furthermore,
the tangential focal data of $M_{\min}$
is given by $\{\theta_{\min},l/(2g_{F}),\{m_{1}^{F},m_{2}^{F}\};2g_{F}\}$
for some $\theta_{\min}\in(0,l/(2g_{F}))$.
By \cite{LiuRa22}, it is known that $\{M_{t}\}_{t}$ converges to $M_{\min}$ as $t \to \infty$.

Under the above setting,
we first describe explicitly 
the exponential decay of the mean curvature $H(t)$,
and $\|A(t)\|^{2}$ as $t\to \infty$.

\begin{fact}[{\cite[Propositions 4.2 and 4.3]{BK}}]\label{thm:BMCF_evaluate}
Assume that $M_{0}=M$ is not minimal.
Then we have the followings:
\begin{enumerate}
\item[(1)] The following equality holds:
\begin{equation}
\lim_{t\to\infty}
H(t)^{2}\exp(\dfrac{1}{\epsilon_{F}}\cdot t)
=\dfrac{m^{F}_{1}+m^{F}_{2}}{4\epsilon_{F}(1-\delta_{F}^{2})}
(\delta_{F}+\cos\bar{\theta}_{F})^{2},
\end{equation}
where
$\delta_{F}$ and $\epsilon_{F}$
denote the constants given in \eqref{eqn:deltaF_epsilonF}.
\item[(2)] The following equality holds:
\begin{equation}
\displaystyle{\lim_{t\to\infty}\|A(t)\|^{2}
=\dfrac{1-2\kappa\epsilon_{F}}{2\epsilon_{F}}
}.
\end{equation}
\end{enumerate}
\end{fact}

Second,
we recall our result on an estimate of the ratio $\|A(t)\|^{2}/H(t)^{2}$ along the BMCF
$\{M_{t}\}_{t\in[0,\infty)}$.

\begin{fact}[{\cite[Theorem 4.5]{BK}}]\label{thm:main3}
Assume that
$M_{0}=M$ is not minimal.
Then we obtain the following assertions:
\begin{enumerate}
\item[(1)] In the case where $\epsilon_{F}=\frac{1}{2\kappa}$ holds,
we have
\begin{equation}
\dfrac{\|A(t)\|^{2}}{H(t)^{2}}\equiv
\dfrac{2}{m_{1}^{F}+m_{2}^{F}}.
\end{equation}
\item[(2)] Otherwise,
the following inequality holds:
\begin{equation}
\dfrac{\|A(t)\|^{2}}{H(t)^{2}}
<\dfrac{4\max\{m_{1}^{F},m_{2}^{F}\}}{(m^{F}_{1}+m^{F}_{2})^{2}}
\cdot\dfrac{1}{(\delta_{F}+\cos\bar{\theta}_{F})^{2}}\cdot
\exp(\dfrac{1}{\epsilon_{F}}\cdot t).
\end{equation}
\end{enumerate}
\end{fact}

Third,
we recall our result on
two types of pinching estimates
for backward mean curvature flows starting from
equifocal hypersurfaces.

\begin{fact}[{\cite[Theorem 4.6]{BK}}]\label{thm:BMCF_evaluate2}
Assume that
$M_{0}=M$ is not minimal.
In the case of $\epsilon_{F} < \frac{1}{2\kappa}$,
there exist $t_{0}>0$ and $c_{1},c_{2}>0$ satisfying the following inequalities:
\begin{equation}\label{eqn:conj1}
c_{2}\cdot\exp(\dfrac{1}{\epsilon_{F}}\cdot t)
\leq 
\dfrac{\|A(t)\|^{2}}{H(t)^{2}}
\leq c_{1}\cdot\exp(\dfrac{1}{\epsilon_{F}}\cdot t)
\quad
(t>t_{0}).
\end{equation}
\end{fact}

\begin{fact}[{\cite[Theorem 4.7]{BK}}]\label{thm:BMCF_evaluate3}
Let $\phi(t)$
denote the squared norm of $\mathring{A}(t)$,
that is,
\begin{equation}
\phi(t):=\|\mathring{A}(t)\|^{2},
\quad
\mathring{A}(t):=A(t)-\dfrac{H(t)}{\dim M_{t}}\mathrm{id}.
\end{equation}
For any $0<h<1$, there exists $c_{0}>0$ such that
one of the following assertions holds:
\begin{enumerate}
\item[(1)] The case of $\delta_{F}=0$:
If $\theta_{F}\in (\theta_{\min}-c_{0},\theta_{\min}+c_{0})$, we have
\begin{equation}\label{eqn:cor_phi1}
\dfrac{1-2\kappa\epsilon_{F}}{2\epsilon_{F}}
\leq
\phi(t)
\leq
\dfrac{1-2\kappa\epsilon_{F}}{2\epsilon_{F}}+h\quad
(t>0).
\end{equation}
\item[(2)] The case of $\delta_{F}>0$: If $\theta_{F}\in (\theta_{\min}-c_{0},\theta_{\min})$,
we have
\begin{equation}\label{eqn:cor_phi2}
\dfrac{1-2\kappa\epsilon_{F}}{2\epsilon_{F}}-h
\leq
\phi(t)
\leq
\dfrac{1-2\kappa\epsilon_{F}}{2\epsilon_{F}}
\quad(t>0).
\end{equation}
If $\theta_{F}\in (\theta_{\min},\theta_{\min}+c_{0})$, we have
\begin{equation}\label{eqn:cor_phi3}
\dfrac{1-2\kappa\epsilon_{F}}{2\epsilon_{F}}\leq
\phi(t)
\leq
\dfrac{1-2\kappa\epsilon_{F}}{2\epsilon_{F}}+h
\quad
(t>0).
\end{equation}
\item[(3)] The case of $\delta_{F}<0$:
If $\theta_{F}\in (\theta_{\min}-c_{0},\theta_{\min})$,
we have
\begin{equation}\label{eqn:cor_phi4}
\dfrac{1-2\kappa\epsilon_{F}}{2\epsilon_{F}}
\leq
\phi(t)
\leq
\dfrac{1-2\kappa\epsilon_{F}}{2\epsilon_{F}}+h
\quad
(t>0).
\end{equation}
If $\theta_{F}\in (\theta_{\min},\theta_{\min}+c_{0})$,
we have
\begin{equation}\label{eqn:cor_phi5}
\dfrac{1-2\kappa\epsilon_{F}}{2\epsilon_{F}}-h
\leq
\phi(t)
\leq
\dfrac{1-2\kappa\epsilon_{F}}{2\epsilon_{F}}
\quad(t>0).
\end{equation}
\end{enumerate}
\end{fact}

\begin{rem}\label{rem:sphere_estimate}
In the case when $M^{n}$ is an isoparametric hypersurface in
the unit sphere $S^{n+1}$ ($\kappa=n$),
its tangential focal data determines the principal curvatures of $M^{n}$.
Namely, $g_{F}$ coincides with 
the number of distinct principal curvatures,
which must be $1,2,3,4$ or $6$ due to M\"unzner (\cite{M80, M81}).
It follows from \cite[(4.9)]{LiuTe20}
that $\epsilon_{F}=(2n\cdot g_{F})^{-1}$ holds.
Thus, the solutions of the BMCF for $M^{n}$
are divided into five classes
according to the decay estimates for $H(t)$.
\end{rem}

\section{Cohomogeneity one Hermann actions and their symmetric triads}\label{sec:cohom1Hermann}

\subsection{Compact symmetric triads and Hermann actions}

\subsubsection{Compact symmetric triads and Hermann actions}

Let $G$ be a compact connected semisimple Lie group.
Let $(G,\theta_{1},\theta_{2})$
be a triplet
consisting of $G$ and
two involutions $\theta_{1},\theta_{2}$ on it.
Let $K_i$ ($i=1,2$) be a closed subgroup of $G$ satisfying
\begin{equation}\label{eqn:symmsub_dfn}
 (G^{\theta_{i}})_{0}\subset K_{i}\subset G^{\theta_{i}}.
\end{equation}
Here $G^{\theta_i}$ denotes the fixed point subgroup of $\theta_{i}$ in
$G$, and $(G^{\theta_i})_{0}$ denotes the connected component containing
the identity element $e$ of $G$.
Then $(G,K_{1},K_{2})$
is called a \emph{compact symmetric triad}.
The homogeneous space $G/K_{i}$ is a symmetric space
of compact type
with respect to a Riemannian metric induced by a bi-invariant Riemannian metric on $G$.

\begin{dfn}
Let $i,j=1,2$.
The natural isometric action
$K_{j}\curvearrowright G/K_{i}$
on the symmetric space $G/K_{i}$ is called the \emph{Hermann action}
associated with $(G,\theta_1,\theta_2)$.
When $\theta_{1}=\theta_{2}$, the Hermann action is nothing but the isotropy action of $G/K_{i}$.
\end{dfn}

Here,
we note that,
in the case when $G/K_{i}$
($i=1,2$)
is simply-connected,
the closed subgroup $K_{i}$ is connected and hence $K_{i}=(G^{\theta_{i}})_{0}$, where $\theta_{i}$ is the involution associated with the
compact symmetric pair $(G,K_{i})$.
Thus, in this case, the compact symmetric triad $(G,K_{1},K_{2})$ can equivalently be described as
$(G,\theta_{1},\theta_{2})$.
For this, we also call $(G,\theta_{1},\theta_{2})$ a compact symmetric triad.
We use the same symbol $\theta_{i}$
as the differential of $\theta_{i}$ at the identity element $e$ of $G$,
which gives an involution
of the Lie algebra $\mathfrak{g}$
of $G$.
Let 
$\mathfrak{g}=\mathfrak{k}_{i}\oplus\mathfrak{p}_{i}$
be the Cartan decomposition of $\mathfrak{g}$ with respect to $\theta_{i}$.
Here, $\mathfrak{k}_{i}$ coincides with the Lie
algebra of $K_{i}$.
In what follows, we consider the $K_{2}$-action on $G/K_{1}$.
The corresponding argument for the $K_{1}$-action on $G/K_{2}$
proceeds in parallel by interchanging the indices $1$ and $2$.
Fix $x\in G/K_{1}$
and let $K_{2}\cdot x$ denote the $K_{2}$-orbit through $x$.
There exists $g\in G$ satisfying $x=gK_{1}$.
Then $g_{*}^{-1}T_{x}(K_{2}\cdot x)$ is a subspace of $\mathfrak{p}_{1}$
which is expressed as follows:
\begin{equation}
g_{*}^{-1}T_{x}(K_{2}\cdot x)
=(d\pi_{1})_{e}(\operatorname{Ad}(g)\mathfrak{k}_2).
\end{equation}
Therefore, we obtain
\begin{equation}\label{eqn:Hermann_nomalsp}
 g_{*}^{-1}T^{\perp}_{x}(K_{2}\cdot x)
 =
 \{X\in\mathfrak{p}_{1}\mid \operatorname{Ad}(g)X\in\mathfrak{p}_{2}\}.
\end{equation}
In particular, at $o:=eK_1$, we have
$T^\perp_o(K_2\cdot o)=\mathfrak{p}_1\cap\mathfrak{p}_2$.

\begin{lem}[{\cite[Lemma 2.4]{IST01}}]
Let $x\in G/K_{1}$, and write $x=gK_{1}$ for some $g\in G$.
Then $g_{*}^{-1}T^{\perp}_{x}(K_{2}\cdot x)$ is a Lie triple system of $\mathfrak{p}_{1}$.
\end{lem}

It is known that Hermann actions are hyperpolar actions
\cite{HPTT}.
Indeed, 
a flat section of a Hermann action is given by
$\pi_{1}(\exp_{G} \mathfrak{b})$
for a maximal abelian subspace $\mathfrak{b}$ of
$\mathfrak{p}_{1}\cap\mathfrak{p}_{2}$.
In particular, the cohomogeneities of the principal orbits of the two Hermann actions
$K_{2}\curvearrowright G/K_{1}$
and
$K_{1}\curvearrowright G/K_{2}$
are equal to $\dim_{\mathbb{R}}\mathfrak{b}$.
The dimension of $\mathfrak{b}$ is called the \emph{rank}
of the compact symmetric triad $(G,K_{1},K_{2})$,
which we write $\mathrm{rank}(G,K_{1},K_{2})$ or $\mathrm{rank}(G,\theta_{1},\theta_{2})$.
The \emph{order} of $(G,K_{1},K_{2})$ is defined as the smallest positive integer
$k$ such that $(\theta_{1}\theta_{2})^{k}$ is the identity map,
which is denoted by $\mathrm{ord}(G,K_{1},K_{2})$ or $\mathrm{ord}(G,\theta_{1},\theta_{2})$.

Matsuki \cite{Ma97, Ma02} defined the following equivalence relation on compact symmetric triads.

\begin{dfn}[]\label{dfn:equiv_Matsuki}
Two compact symmetric triads
$(G,\theta_{1},\theta_{2})$ and
$(G,\theta_{1}',\theta_{2}')$
are said to be \emph{isomorphic} if there exist an automorphism $\varphi$ of $G$
and an inner automorphism $\tau$ of $G$ satisfying
\begin{equation}\label{eqn:Matsuki_sim}
\theta_{1}'=\varphi\theta_{1}\varphi^{-1},\quad
\theta_{2}'=\tau\varphi\theta_{2}\varphi^{-1}\tau^{-1}.
\end{equation}
In this case, we write
$(G,\theta_{1},\theta_{2})\sim(G,\theta_{1}',\theta_{2}')$,
and $\sim$ defines an equivalence relation on the set of compact symmetric triads.
As the Lie algebra version of $\sim$,
two compact symmetric triads
$(G,\theta_{1},\theta_{2})$ and
$(G,\theta_{1}',\theta_{2}')$
are said to be \emph{locally isomorphic} if there exist an automorphism $\varphi$ of $\mathfrak{g}$
and an inner automorphism $\tau$ of $\mathfrak{g}$ satisfying
\begin{equation}\label{eqn:Matsuki_sim_local}
\theta_{1}'=\varphi\theta_{1}\varphi^{-1},\quad
\theta_{2}'=\tau\varphi\theta_{2}\varphi^{-1}\tau^{-1}.
\end{equation}
\end{dfn}
A stronger equivalence relation $\equiv$ than $\sim$ on the set of compact symmetric triads
is defined.
Namely, 
we write
$(G,\theta_{1},\theta_{2})\equiv(G,\theta_{1}',\theta_{2}')$
if there exists an automorphism $\varphi$ of $G$ satisfying
$\theta_{i}'=\varphi\theta_{i}\varphi^{-1}$ for $i=1,2$.
By definition,
$(G,\theta_{1},\theta_{2})\equiv(G,\theta_{1}',\theta_{2}')$
yields $(G,\theta_{1},\theta_{2})\sim(G,\theta_{1}',\theta_{2}')$.
In a similar manner, we can define the Lie algebra version of $\equiv$.

\begin{lem}\label{lem:rank_equiv}
Assume that $(G,\theta_{1},\theta_{2})\equiv(G,\theta_{1}',\theta_{2}')$ at the Lie algebra level.
Then we have
$\mathrm{rank}(G,\theta_{1},\theta_{2})
=\mathrm{rank}(G,\theta_{1}',\theta_{2}')$
and
$\mathrm{ord}(G,\theta_{1},\theta_{2})=
\mathrm{ord}(G,\theta_{1}',\theta_{2}')$.
\end{lem}

\begin{proof}
Let $\varphi$ be an automorphism of $\mathfrak{g}$ satisfying
$\theta_{i}'=\varphi\theta_{i}\varphi^{-1}$
for $i=1,2$.
Let
$\mathfrak{g}=\mathfrak{k}_{i}'\oplus\mathfrak{p}_{i}'$
be the Cartan decomposition of $\mathfrak{g}$ with respect to $\theta_{i}'$.
Then
$\mathfrak{p}_{i}'=\varphi(\mathfrak{p}_{i})$ holds.
In particular, we get
$\mathfrak{p}_{1}'\cap\mathfrak{p}_{2}'
=\varphi(\mathfrak{p}_{1}\cap\mathfrak{p}_{2})$.
Hence we obtain
$\mathrm{rank}(G,\theta_{1},\theta_{2})=
\mathrm{rank}(G,\theta_{1}',\theta_{2}')$.
Furthermore,
it follows from
$\theta_{1}'\theta_{2}'=
\varphi\theta_{1}\theta_{2}\varphi^{-1}$
that
$\mathrm{ord}(G,\theta_{1},\theta_{2})
=\mathrm{ord}(G,\theta_{1}',\theta_{2}')$ holds.
Thus, we have completed the proof.
\end{proof}

\begin{pro}\label{pro:sim_rank}
If two compact symmetric triads
$(G,\theta_{1},\theta_{2})$ and
$(G,\theta_{1}',\theta_{2}')$
are locally isomorphic in the sense of Definition \ref{dfn:equiv_Matsuki},
then the ranks of
$(G,\theta_{1},\theta_{2})$ and
$(G,\theta_{1}',\theta_{2}')$
coincide.
\end{pro}

\begin{proof}
By the assumption there exist an automorphism $\varphi$ of $\mathfrak{g}$ and an inner automorphism $\tau$ of $\mathfrak{g}$
satisfying
$\theta_{1}'=\varphi\theta_{1}\varphi^{-1}$
and $\theta_{2}'=\varphi\tau\theta_{2}\tau^{-1}\varphi^{-1}$.
Then there exists $g\in G$ satisfying $\tau=\mathrm{Ad}(g)$,
where $\mathrm{Ad}$ denotes the adjoint representation of $G$.
By Hermann's theorem, there exist $k_{i}\in K_{i}$ and $Z\in\mathfrak{b}$
satisfying
$g=k_{1}\exp_{G}(Z)k_{2}$.
Hence we have
$\mathrm{Ad}(g)=\mathrm{Ad}(k_{1})\mathrm{Ad}(\exp_{G}(Z))\mathrm{Ad}(k_{2})$.
For $g\in G$,
we define the inner automorphism
$\tau_{g}$ by $\tau_{g}(h)=ghg^{-1}$
($h\in G$).
Since
$\mathrm{Ad}(k_{i})\theta_{i}
=\theta_{i}\mathrm{Ad}(k_{i})$
holds, we have $(G,\theta_{1}',\theta_{2}')\equiv
(G,\theta_{1},
\tau_{\exp_{G}(Z)}\theta_{2}\tau_{\exp_{G}(Z)}^{-1})$ at the Lie algebra level.
We put $\theta_{2}'':=\tau_{\exp_{G}(Z)}\theta_{2}\tau_{\exp_{G}(Z)}^{-1}$.
Then Lemma \ref{lem:rank_equiv} yields
\begin{equation}\label{eqn:sim_rank1}
\mathrm{rank}(G,\theta_{1}',\theta_{2}')
=
\mathrm{rank}(G,\theta_{1},\theta_{2}'').
\end{equation}
On the other hand, let $\mathfrak{g}=\mathfrak{k}_{2}''\oplus\mathfrak{p}_{2}''$
be the Cartan decomposition with respect to $\theta_{2}''$.
Then
$\mathfrak{p}_{1}\cap\mathfrak{p}_{2}''
=\mathfrak{p}_{1}\cap
(e^{\mathrm{ad}(Z)}\mathfrak{p}_{2})$.
It follows from this expression that the maximal abelian subspace
$\mathfrak{b}$ of $\mathfrak{p}_{1}\cap\mathfrak{p}_{2}$
is also a maximal abelian subspace of
$\mathfrak{p}_{1}\cap\mathfrak{p}_{2}''$.
Hence
\begin{equation}\label{eqn:sim_rank2}
\mathrm{rank}(G,\theta_{1},\theta_{2})
=
\mathrm{rank}(G,\theta_{1},\theta_{2}'')
\end{equation}
holds.
The assertion follows from
\eqref{eqn:sim_rank1}
and
\eqref{eqn:sim_rank2}.
\end{proof}

Here, we note that even if
$(G,\theta_{1},\theta_{2})\sim(G,\theta_{1}',\theta_{2}')$,
their orders do not necessarily coincide.

In general, let $L_{i}$ be a Lie group acting isometrically on a Riemannian manifold $X_{i}$
($i=1,2$).
Two isometric actions
$L_{1}\curvearrowright X_{1}$
and
$L_{2}\curvearrowright X_{2}$
are said to be \emph{isomorphic} if there exist an isometry
$\Phi:X_{1}\to X_{2}$
and an isomorphism of Lie groups
$f:L_{1}\to L_{2}$
satisfying
$\Phi(g\cdot x)=f(g)\cdot \Phi(x)$
for all $g\in L_{1}$ and $x\in X_{1}$.
This relation defines an equivalence relation on the set of isometric actions on Riemannian manifolds.
On the other hand, the notion of orbit equivalence is a weaker equivalence relation than isomorphism.
Namely, two isometric actions
$L_{1}\curvearrowright X_{1}$
and
$L_{2}\curvearrowright X_{2}$
are said to be \emph{orbit equivalent} if there exists an isometry
$\Phi:X_{1}\to X_{2}$
satisfying
$\Phi(L_{1}\cdot x)=L_{2}\cdot \Phi(x)$
for all $x\in X_{1}$.

\begin{lem}\label{lem:2Hermann_isom}
Let
$K'_{2}\curvearrowright G/K'_{1}$
denote the Hermann action associated with
$(G,\theta_{1}',\theta_{2}')$.
Assume that $K_{i}$ and $K'_{i}$ are connected for $i=1,2$.
If
$(G,\theta_{1},\theta_{2})\sim(G,\theta_{1}',\theta_{2}')$,
then the actions
$K_{2}\curvearrowright G/K_{1}$
and
$K'_{2}\curvearrowright G/K'_{1}$
are isomorphic.
In particular, they are orbit equivalent.
\end{lem}

\begin{proof}
Let $\varphi$ be an automorphism of $G$ and $\tau$ an inner automorphism of $G$
satisfying \eqref{eqn:Matsuki_sim}, and take $a \in G$ with $\tau = \tau_{a}$.
Since $\theta_{1}' = \varphi\theta_{1}\varphi^{-1}$ and
$\theta_{2}' = \tau\varphi\theta_{2}\varphi^{-1}\tau^{-1}$,
we have $\varphi(G^{\theta_{1}}) = G^{\theta_{1}'}$ and
$\tau\varphi(G^{\theta_{2}}) = G^{\theta_{2}'}$.
Since $K_{i}$ and $K_{i}'$ are connected, we obtain
$K_{1}' = \varphi(K_{1})$ and $K_{2}' = \tau\varphi(K_{2})$.
In particular, $\tau\varphi|_{K_{2}}: K_{2} \to K_{2}'$ is an isomorphism of Lie groups.
Define $\Phi : G/K_{1} \to G/K_{1}'$ by
\begin{equation}
\Phi(gK_{1}) := a\,\varphi(g)K_{1}' \quad (g\in G).
\end{equation}
This is well-defined because $\varphi(K_{1}) = K_{1}'$, and it is an isometry,
since it is the composition of the isometry $gK_{1} \mapsto \varphi(g)K_{1}'$
and the left translation by $a$.
For any $k\in K_{2}$ and $g\in G$, we have
\begin{equation}
\Phi(k \cdot gK_{1})
= a\varphi(k)\varphi(g)K_{1}'
= (a\varphi(k)a^{-1})a\varphi(g)K_{1}'
= \tau\varphi(k) \cdot \Phi(gK_{1}).
\end{equation}
Hence the actions $K_{2} \curvearrowright G/K_{1}$ and
$K_{2}' \curvearrowright G/K_{1}'$ are isomorphic via
$(\Phi, \tau\varphi|_{K_{2}})$.
Since $\Phi(K_{2} \cdot x) = K_{2}' \cdot \Phi(x)$ holds for all $x \in G/K_{1}$,
they are also orbit equivalent.
\end{proof}

\begin{proof}
Let $\varphi$ be an automorphism of $G$ and let $\tau$ be an inner automorphism of $G$
satisfying \eqref{eqn:Matsuki_sim}.
Then we have
$K_{1}'=\varphi(K_{1})$
and
$K_{2}'=\tau\varphi(K_{2})$.
In particular,
$\varphi|_{K_{2}}$
is an isomorphism from $K_{2}$ onto $\tau^{-1}(K_{2}')$.
Define an isometry
$\Phi:G/K_{1}\to G/K_{1}'$
by
$\Phi(gK_{1}):=\varphi(g)K_{1}'$
for $g\in G$.
Then, for any $k\in K_{2}$, we have
$\Phi(k\cdot gK_{1})
=\varphi(kg)K_{1}'
=\varphi(k)\varphi(g)K_{1}'
=\varphi(k)\cdot \Phi(gK_{1})$.
Hence the actions
$K_{2}\curvearrowright G/K_{1}$
and
$\tau^{-1}(K'_{2})\curvearrowright G/K'_{1}$
are isomorphic.
Clearly, the identity map
$\mathrm{id}_{G/K'_{1}}:G/K'_{1}\to G/K'_{1}$
is an isometry, and
$\tau^{-1}|_{K'_{2}}:K'_{2}\to\tau^{-1}(K'_{2})$
is an isomorphism of Lie groups.
Therefore, the actions
$K'_{2}\curvearrowright G/K'_{1}$
and
$\tau^{-1}(K'_{2})\curvearrowright G/K'_{1}$
are isomorphic.
Therefore,
$K_{2}\curvearrowright G/K_{1}$
and
$K'_{2}\curvearrowright G/K'_{1}$
are isomorphic.
Thus the assertion follows.
\end{proof}

\subsubsection{Canonical forms of compact symmetric triads}\label{sec:can_CST}

The submanifold geometry of orbits of Hermann actions is concerned with properties
which are invariant under orbit equivalence.
By Lemma \ref{lem:2Hermann_isom}, an isomorphism class of compact symmetric triads
allows one to replace a representative without changing the orbit equivalence class
of the associated Hermann action.
It important to select a canonical representative of compact symmetric triads
for the relevant computations.
The notion of a canonical compact symmetric triad was introduced by the first author and Ikawa
\cite{BI24}. In this section, we recall its definition and basic properties.

\begin{dfn}[{\cite[Definition 6.1]{BI24}}]\label{dfn:canonical}
A compact symmetric triad $(G,\theta_{1},\theta_{2})$ is said to be \emph{canonical}
if there exists a maximal abelian subalgebra $\mathfrak{t}$ of $\mathfrak{g}$
satisfying the following conditions:
\begin{enumerate}
\item[(1)] For $i=1,2$, $\mathfrak{t}\cap\mathfrak{p}_{i}$
is a maximal abelian subspace of $\mathfrak{p}_{i}$.
In particular, $\mathfrak{t}$ is invariant under the action of $\theta_{i}$.
\item[(2)] Let $\Delta(\mathfrak{g},\mathfrak{t})$ be the root system with respect to $\mathfrak{t}$,
and put $\sigma_{i}:=(-\theta_{i})|_{\mathfrak{t}}$.
Then there exists a linear ordering $<$ on $\mathfrak{t}$
such that $\sigma_{i}(\alpha)>0$ ($i=1,2$) holds for $\alpha\in\Delta(\mathfrak{g},\mathfrak{t})^{+}-\Delta^{i}_{0}$, where $\Delta(\mathfrak{g},\mathfrak{t})^{+}$ denotes the set of positive roots
in $\Delta(\mathfrak{g},\mathfrak{t})$ with respect to $<$,
and $\Delta^{i}_{0}=\Delta(\mathfrak{g},\mathfrak{t})\cap\mathfrak{k}_{i}$.
\item[(3)] $\mathrm{ord}(G,\theta_{1},\theta_{2})$ is equal to the order of the restriction of
$\theta_{1}\theta_{2}$ to $\mathfrak{t}$.
\end{enumerate}
In this case, $\mathfrak{t}$ is said to be canonical with respect to $(G,\theta_{1},\theta_{2})$.
\end{dfn}

Definition \ref{dfn:canonical}, (2), means that the Satake diagram $S_{1}$
of $(G,\theta_{1})$ and the Satake diagram $S_{2}$ of $(G,\theta_{2})$
can be drawn from the common Dynkin diagram.
Then, the pair $(S_{1},S_{2})$ is called the \emph{double Satake diagram}
of $(G,\theta_{1},\theta_{2})$.
On the other hand, $\theta_{1}\theta_{2}|_{\mathfrak{t}}$ induces a permutation of the root system
$\Delta$.
Therefore, Definition \ref{dfn:canonical}, (3), implies that the order of
$(G,\theta_{1},\theta_{2})$ is finite.

\begin{ex}
(1) For an involution $\theta$ of $G$, the compact symmetric triad
$(G,\theta,\theta)$ is canonical.
(2) For two involutions $\theta_{1}$ and $\theta_{2}$ of $G$, we write
$\theta_{1}\sim\theta_{2}$ if there exists an inner automorphism $\tau$ of $G$ such that
$\theta_{2}=\tau\theta_{1}\tau^{-1}$.
Assume that two involutions $\theta_{1}$ and $\theta_{2}$ of $G$ satisfy
$\theta_{1}\theta_{2}=\theta_{2}\theta_{1}$
and
$\theta_{1}\not\sim\theta_{2}$.
Then it was shown in \cite{BI24} that
$(G,\theta_{1},\theta_{2})$ is canonical.
\end{ex}

\begin{thm}[{\cite[Theorems 6.6 and 6.12, (1)]{BI24}}]\label{thm:BK24_thm}
Assume that $G$ is simple.
Then the following assertions hold:
\begin{enumerate}
\item[(1)] For any compact symmetric triad $(G,\theta_{1},\theta_{2})$,
there exists a canonical compact symmetric triad
$(G,\theta_{1}',\theta_{2}')$ which is isomorphic to $(G,\theta_{1},\theta_{2})$.
\item[(2)] Let $(G,\theta_{1},\theta_{2})$ be a canonical compact symmetric triad,
and $\mathfrak{t}$ be a maximal abelian subalgebra of $\mathfrak{g}$
which is canonical with respect to
$(G,\theta_{1},\theta_{2})$.
Then
$\mathfrak{b}:=\mathfrak{t}\cap(\mathfrak{p}_{1}\cap\mathfrak{p}_{2})$
is a maximal abelian subspace of
$\mathfrak{p}_{1}\cap\mathfrak{p}_{2}$.
\end{enumerate}
\end{thm}


\subsection{Data for principal orbits}\label{sec:data_principalorbit}

In this subsection, we first describe the restricted root data of principal orbits
of cohomogeneity one Hermann actions in Subsection \ref{sec:r_data}.
This data is defined by the eigenvalues of the normal Jacobi operator
explained in Subsection \ref{sec:eigen_equihyper}.
Second, following \cite{Ikawa11} and \cite{Oh23},
we describe the principal curvatures of principal orbits of cohomogeneity one Hermann actions
by using symmetric triads with multiplicities associated with compact symmetric triads.
For this purpose, in Subsection \ref{sec:STm}, we recall symmetric triads with multiplicities
for compact symmetric triads.
The descriptions of the principal curvatures of principal orbits of cohomogeneity one Hermann actions are given in Subsections \ref{sec:p-data}.

In what follows, let $G$ be a compact connected simple Lie group.
Assume that 
$G/K_{i}$ ($i=1,2$) is simply-connected.
We also call $(G,K_{1},K_{2})$
a canonical compact symmetric triad
if $(G,\theta_{1},\theta_{2})$ is canonical in the sense of Definition \ref{dfn:canonical},
where $\theta_{i}$ denotes the involution associated with $(G,K_{i})$.
By Theorem \ref{thm:BK24_thm}, (1),
without loss of generality we may assume that $(G,K_{1},K_{2})$ is
a canonical compact symmetric triad
with rank one.
Furthermore, we take $\mathfrak{b}$
given by Theorem \ref{thm:BK24_thm}, (2)
as a maximal abelian subspace
of $\mathfrak{p}_{1}\cap\mathfrak{p}_{2}$.

\subsubsection{Restricted root data for principal orbits}\label{sec:r_data}
Let $M$ be a principal orbit of the action of $K_{2}$ on $G/K_{1}$.
Let $\kappa_{1}$ be the Einstein constant of $G/K_{1}$.
We may assume without loss of generality that a base point $x\in M$
lies on the flat section $\pi_{1}(\exp_{G}\mathfrak{b})$.
There exists $Z\in\mathfrak{b}$ satisfying
$x=\pi_{1}(\exp_{G} Z)$.
For simplicity, we put $g=\exp_{G} Z$.
Then $T^{\perp}_{x}M=g_{*}\mathfrak{b}$.
Fix a unit vector $\xi\in\mathfrak{b}$.
Since $(G,K_{1},K_{2})$ has rank one,
we have $\mathfrak{b}=\mathbb{R}\xi$.
Then $g_{*}\xi$ is a unit vector in $T^{\perp}_{x}M$.
Since $T^{\perp}M$ is globally flat,
$g_{*}\xi$ extends to a unit normal vector field on $M$,
which is denote by $\bs{\xi}$.
We call $\bs{\xi}$ the unit normal vector field on $M$
induced from $\xi\in\mathfrak{b}$.
The notations $\Delta$, $\Lambda^{R}$
and $\tilde{m}_{\mu}$
($\mu\in\Lambda^{R}$)
are as in Subsection \ref{sec:eigen_equihyper}.

\begin{pro}[{\cite[PROPOSITION 1]{Ma97}}]\label{pro:Matsuki}
The set $\Delta$ satisfies the axioms of a root system.
\end{pro}

The set $\Delta$ is also called the \emph{restricted root system}
of $(G,\theta_{1},\theta_{2})$ with respect to $\mathfrak{b}$.
Since $\Delta$ has rank one,
$\Delta$ is either of $A_{1}$-type or
$BC_{1}$-type.
Here, in the case when $\Delta$ is of $A_{1}$-type,
we write $\Delta=\{\pm \mu\}$,
and in the case when $\Delta$ is of $BC_{1}$-type, we also write
$\Delta=\{\pm \mu,\pm 2\mu\}$.

\begin{pro}
The length $l$ of the normal circle of $M$ as an equifocal hypersurface
has the following expression:
\begin{equation}\label{eqn:l_formula_tildem}
l=\pi \cdot l_{\mu} \sqrt{\dfrac{\tilde{m}_{\mu}+4\tilde{m}_{2\mu}}{\kappa_{1}}},
\end{equation}
where $l_{\mu}$ denotes
the integer given in Proposition \ref{pro:normal_circle_period}
and $\tilde{m}_{2\mu}$ is understood to be zero
in the case when $\Delta$ is of type $A_1$.
\end{pro}

\begin{proof}
By using \eqref{eqn:decomp_p_delta}, we have
$\kappa_{1}=
\operatorname{Tr}_{T_{eK}(G/K_{1})}(-\operatorname{ad}_{\xi}^{2})
=\tilde{m}_{\mu}\mu(\xi)^{2}+4\tilde{m}_{2\mu}\mu(\xi)^{2}$,
from which we get
\begin{equation}\label{eqn:muxi}
\mu(\xi)=\sqrt{\dfrac{\kappa_{1}}{\tilde{m}_{\mu}+4\tilde{m}_{2\mu}}}.
\end{equation}
Thus, by using \eqref{eqn:Tx_period} and \eqref{eqn:muxi},
we obtain the assertion.
\end{proof}

\subsubsection{Symmetric triads with multiplicities for compact symmetric triads}\label{sec:STm}

We recall symmetric triads with multiplicities for rank one compact symmetric triads
$(G,K_{1},K_{2})$.
We begin by giving a decomposition of
$\mathfrak{p}_{1}(\mathfrak{b},\mu)$
which is compatible with $\theta_{1}\theta_{2}$,
where $\mathfrak{b}$ is  a maximal abelian subspace of $\mathfrak{p}_{1}\cap\mathfrak{p}_{2}$
and $\mathfrak{p}_{1}(\mathfrak{b},\mu)$
denotes the subspace of $\mathfrak{p}_{1}$
defined as in \eqref{eqn:pbmu_dfn}.
Since $\theta_{1}\theta_{2}$ may have complex eigenvalues,
we rewrite $\mathfrak{p}_{1}(\mathfrak{b},\mu)$
in terms of the complexification of $\mathfrak{g}$.
We put $k:=\mathrm{ord}(G,\theta_{1},\theta_{2})$,
which is finite since $(G,\theta_{1},\theta_{2})$ is canonical.
Let $C_{k}(\subset \mathbb{C})$ denote the set of all $k$-th roots of unity,
and $\mathfrak{g}^{\mathbb{C}}$ denote the complexification of $\mathfrak{g}$.
We use the same symbols for the complexifications of automorphisms of $\mathfrak{g}$.
For $\mu\in\mathfrak{b}^{*}$,
the subspace $\mathfrak{g}^{\mathbb{C}}(\mathfrak{b},\mu)$
of $\mathfrak{g}^{\mathbb{C}}$ is defined as follows:
\begin{equation}
\mathfrak{g}^{\mathbb{C}}(\mathfrak{b},\mu)
:=\{X\in\mathfrak{g}^{\mathbb{C}}\mid \mathrm{ad}_{\xi}X=\sqrt{-1}\mu(\xi)X\},
\end{equation}
where $\xi$ is a unit vector of $\mathfrak{b}$.

\begin{lem}\label{lem:gamu_p}
Let $i=1,2$.
Let $\mu\in\mathfrak{b}^{*}$.
Then
$\overline{\mathfrak{g}^{\mathbb{C}}(\mathfrak{b},\mu)}=
\mathfrak{g}^{\mathbb{C}}(\mathfrak{b},-\mu)$
and
$\theta_{i}\mathfrak{g}^{\mathbb{C}}(\mathfrak{b},\mu)=\mathfrak{g}^{\mathbb{C}}(\mathfrak{b},-\mu)$
hold.
Here, $\overline{\;\cdot\;}$ denotes the conjugation of $\mathfrak{g}^{\mathbb{C}}$
with respect to $\mathfrak{g}$.
\end{lem}

\begin{proof}
Let $X$ be in $\mathfrak{g}^{\mathbb{C}}(\mathfrak{b},\mu)$.
For any $Y\in\mathfrak{b}$, we have
$\mathrm{ad}_{Y}\overline{X}=\overline{\mathrm{ad}_{Y}X}=
-\sqrt{-1}\mu(Y)\overline{X}$.
Hence we get $\overline{X}\in\mathfrak{g}^{\mathbb{C}}(\mathfrak{b},-\mu)$,
from which
$\overline{\mathfrak{g}^{\mathbb{C}}(\mathfrak{b},\mu)}
\subset\mathfrak{g}^{\mathbb{C}}(\mathfrak{b},-\mu)$
holds.
We show the converse.
Let $X$ be in $\mathfrak{g}^{\mathbb{C}}(\mathfrak{b},-\mu)$.
By the same computation, we have
$\overline{X}\in\mathfrak{g}^{\mathbb{C}}(\mathfrak{b},\mu)$,
and hence
$X\in\overline{\mathfrak{g}^{\mathbb{C}}(\mathfrak{b},\mu)}$.
Therefore,
$\overline{\mathfrak{g}^{\mathbb{C}}(\mathfrak{b},\mu)}
=\mathfrak{g}^{\mathbb{C}}(\mathfrak{b},-\mu)$.
By the same argument, one can verify that
$\theta_{i}\mathfrak{g}^{\mathbb{C}}(\mathfrak{b},\mu)
=\mathfrak{g}^{\mathbb{C}}(\mathfrak{b},-\mu)$.
\end{proof}

We put
$V(\mathfrak{b},\mu):=(\mathfrak{g}^{\mathbb{C}}(\mathfrak{b},\mu)+
\overline{\mathfrak{g}^{\mathbb{C}}(\mathfrak{b},\mu)})\cap\mathfrak{g}$,
which is a direct sum if $\mu\neq0$.
By Lemma \ref{lem:gamu_p}, $V(\mathfrak{b},\mu)$ is invariant under
the action of $\theta_{i}$
and hence we have
$V(\mathfrak{b},\mu)=(V(\mathfrak{b},\mu)\cap\mathfrak{k}_{i})\oplus
(V(\mathfrak{b},\mu)\cap\mathfrak{p}_{i})$.

\begin{lem}
We put
$\Delta':=\{\mu\in\mathfrak{b}^{*}\mid \mathfrak{g}^{\mathbb{C}}(\mathfrak{b},\mu)\neq\{0\}\}-\{0\}$.
For any $\mu\in\Delta'\cup\{0\}$, 
we have
$\mathfrak{p}_{1}(\mathfrak{b},\mu)
=V(\mathfrak{b},\mu)\cap\mathfrak{p}_{1}$.
Furthermore,
$\Delta=\Delta'$ holds.
\end{lem}

\begin{proof}
We first show that
$\mathfrak{p}_{1}(\mathfrak{b},\mu)
=V(\mathfrak{b},\mu)\cap\mathfrak{p}_{1}$
for $\mu\in\Delta'\cup\{0\}$.
Let $X$ be in $V(\mathfrak{b},\mu)\cap\mathfrak{p}_{1}$.
Since $X\in V(\mathfrak{b},\mu)$, there exist
$X'\in\mathfrak{g}^{\mathbb{C}}(\mathfrak{b},\mu)$
and
$X''\in\mathfrak{g}^{\mathbb{C}}(\mathfrak{b},-\mu)$
satisfying $X=X'+X''$.
Hence, for any $Y\in\mathfrak{b}$, a direct calculation shows
\begin{equation}
\mathrm{ad}_{Y}^{2}X
=\mathrm{ad}_{Y}^{2}X'+\mathrm{ad}_{Y}^{2}X''
=(\sqrt{-1}\mu(Y))^{2}X'+(-\sqrt{-1}\mu(Y))^{2}X''=-\mu(Y)^{2}X.
\end{equation}
Thus $X\in\mathfrak{p}_{1}(\mathfrak{b},\mu)$, and hence
$V(\mathfrak{b},\mu)\cap\mathfrak{p}_{1}
\subset\mathfrak{p}_{1}(\mathfrak{b},\mu)$ holds.
Conversely,
for
$X\in\mathfrak{p}_{1}(\mathfrak{b},\mu)(\subset\mathfrak{g}^{\mathbb{C}})$,
it is verified that there exist
$X_{0}\in\mathfrak{g}^{\mathbb{C}}(\mathfrak{b},0)$
and
$X_{\mu'}\in\mathfrak{g}^{\mathbb{C}}(\mathfrak{b},\mu')$
$(\mu'\in\Delta')$
satisfying $X=X_{0}+\sum_{\mu'\in\Delta'}X_{\mu'}$
by using the following decomposition of $\mathfrak{g}^{\mathbb{C}}$:
\begin{equation}
\mathfrak{g}^{\mathbb{C}}=
\mathfrak{g}^{\mathbb{C}}(\mathfrak{b},0)
\oplus\bigoplus_{\mu\in\Delta'}\mathfrak{g}^{\mathbb{C}}(\mathfrak{b},\mu).
\end{equation}
Therefore, for any $Y\in\mathfrak{b}$, we obtain
\begin{equation}\label{eqn:p1_V1}
\mathrm{ad}_{Y}^{2}X
=\mathrm{ad}_{Y}^{2}X_{0}+\sum_{\mu'\in\Delta'}\mathrm{ad}_{Y}^{2}X_{\mu'}
=-\sum_{\mu'\in\Delta'}\mu'(Y)^{2}X_{\mu'}.
\end{equation}
On the other hand, by the definition of $\mathfrak{p}_{1}(\mathfrak{b},\mu)$, we obtain
\begin{equation}\label{eqn:p1_V2}
\mathrm{ad}_{Y}^{2}X=-\mu(Y)^{2}X
=-\mu(Y)^{2}X_{0}-\mu(Y)^{2}\sum_{\mu'\in\Delta'}X_{\mu'}.
\end{equation}
By comparing \eqref{eqn:p1_V1} with \eqref{eqn:p1_V2}, we find that
$X=X_{\mu}+X_{-\mu}$,
from which $X\in V(\mathfrak{b},\mu)$ holds.
Therefore,
$\mathfrak{p}_{1}(\mathfrak{b},\mu)
=V(\mathfrak{b},\mu)\cap\mathfrak{p}_{1}$
has been shown.
This equality implies $\Delta=\Delta'$.
Indeed, if $\mu\in\Delta'$, then
$\mathfrak{p}_{1}(\mathfrak{b},\mu)
=V(\mathfrak{b},\mu)\cap\mathfrak{p}_{1}\neq\{0\}$.
This yields $\mu$ is in $\Delta$,
from which $\Delta'\subset\Delta$ holds.
In the same way,
we have $\Delta'\supset\Delta$.
Thus, we have completed the proof.
\end{proof}

We next explain symmetric triads with multiplicities for compact symmetric triads.
Since $\mathfrak{g}^{\mathbb{C}}(\mathfrak{b},\mu)$ is invariant under $\theta_{1}\theta_{2}$,
we have the following direct sum decomposition:
\begin{equation}\label{eqn:gamu_gamuep}
\mathfrak{g}^{\mathbb{C}}(\mathfrak{b},\mu)=\bigoplus_{\epsilon\in C_{k}}
\mathfrak{g}^{\mathbb{C}}(\mathfrak{b},\mu,\epsilon),
\end{equation}
where we put
$\mathfrak{g}^{\mathbb{C}}(\mathfrak{b},\mu,\epsilon):=
\{X\in\mathfrak{g}^{\mathbb{C}}(\mathfrak{b},\mu)\mid \theta_{1}\theta_{2}X=\epsilon X\}$
for $\epsilon\in C_{k}$.

\begin{lem}\label{lem:gamuep_p}
Let $\mu\in\Delta$ and $\epsilon\in C_{k}$.
Then we have
$\theta_{i}\mathfrak{g}^{\mathbb{C}}(\mathfrak{b},\mu,\epsilon)
=\mathfrak{g}^{\mathbb{C}}(\mathfrak{b},-\mu,\epsilon^{-1})$
and
$\overline{\mathfrak{g}^{\mathbb{C}}(\mathfrak{b},\mu,\epsilon)}
=
\mathfrak{g}^{\mathbb{C}}(\mathfrak{b},-\mu,\epsilon^{-1})$.
\end{lem}

\begin{proof}
Suppose that $X\in\mathfrak{g}^{\mathbb{C}}$ satisfies
$\theta_{1}\theta_{2}X=\epsilon X$.
Since $(\theta_{1}\theta_{2})^{k}=1$, we have
$(\theta_{1}\theta_{2})^{k-1}=\theta_{2}\theta_{1}$.
Hence
$\theta_{1}\theta_{2}(\theta_{1}(X))
=\theta_{1}(\epsilon^{k-1}X)
=\epsilon^{-1}\theta_{1}(X)$.
By combining this with Lemma \ref{lem:gamu_p}, we obtain the assertion.
\end{proof}

For $\mu\in\Delta$ and $\epsilon\in C_{k}$,
we put
$V(\mathfrak{b},\mu,\epsilon):=
(\mathfrak{g}^{\mathbb{C}}(\mathfrak{b},\mu,\epsilon)\oplus
\overline{\mathfrak{g}^{\mathbb{C}}(\mathfrak{b},\mu,\epsilon)})\cap\mathfrak{g}$.
Then, by using \eqref{eqn:gamu_gamuep}
and Lemma \ref{lem:gamuep_p}, we obtain the following decomposition:
\begin{equation}\label{eqn:Vamu_Vamuep}
V(\mathfrak{b},\mu)
=\bigoplus_{\epsilon\in C_{k}}
(\mathfrak{g}^{\mathbb{C}}(\mathfrak{b},\mu,\epsilon)
\oplus
\overline{\mathfrak{g}^{\mathbb{C}}(\mathfrak{b},\mu,\epsilon)})\cap\mathfrak{g}
=\bigoplus_{\epsilon\in C_{k}}V(\mathfrak{b},\mu,\epsilon).
\end{equation}
Put
$\mathfrak{p}_{1}(\mathfrak{b},\mu,\epsilon)
:=V(\mathfrak{b},\mu,\epsilon)\cap\mathfrak{p}_{1}$.
Since, by Lemma \ref{lem:gamuep_p},
$V(\mathfrak{b},\mu,\epsilon)$ is invariant under $\theta_{i}$,
it follows from \eqref{eqn:Vamu_Vamuep} that the following direct sum decomposition holds:
\begin{equation}\label{eqn:p1bmu_decomp}
\mathfrak{p}_{1}(\mathfrak{b},\mu)
=\bigoplus_{\epsilon\in C_{k}}
\mathfrak{p}_{1}(\mathfrak{b},\mu,\epsilon)
\end{equation}
For $\epsilon\in C_{k}$,
the
subset $\Delta_{\epsilon}$ of $\Delta$ is defined as follows:
\begin{equation}\label{eqn:dfn_delta_epsilon}
\Delta_{\epsilon}
:=\{\mu\in\Delta\mid \mathfrak{p}_{1}(\mathfrak{b},\mu,\epsilon)\neq\{0\}\}.
\end{equation}
We put $\Lambda^{R}_{\epsilon}:=\Lambda^{R}\cap\Delta_{\epsilon}$.
By \eqref{eqn:decomp_p_delta} and \eqref{eqn:p1bmu_decomp} we have the following direct sum decomposition of $\mathfrak{p}_{1}$:
\begin{equation}\label{eqn:p1_decomp2}
\mathfrak{p}_{1}=
\mathfrak{p}_{1}(\mathfrak{b},0)\oplus
\bigoplus_{\epsilon\in C_{k}}\bigoplus_{\mu\in\Lambda^{R}_{\epsilon}}
\mathfrak{p}_{1}(\mathfrak{b},\mu,\epsilon).
\end{equation}
For $\epsilon\in C_{k}$ and $\mu\in\Delta_{\epsilon}$,
the dimension of the real vector space
$\mathfrak{p}_{1}(\mathfrak{b},\mu,\epsilon)$
is called the \emph{multiplicity} of the pair $(\mu,\epsilon)$,
which is denoted by $m(\mu,\epsilon)$.
We note that the following equality holds:
\begin{equation}\label{eqn:tildem_m}
\tilde{m}_{\mu}=\sum_{\epsilon\in C_{k}}m(\mu,\epsilon).
\end{equation}

\begin{dfn}\label{dfn:SMm_dfn}
Retain the notation above.
We call
$\mathcal{S}:=(\Delta,\{\Delta_{\epsilon}\}_{\epsilon\in C_{k}})$
the \emph{symmetric triad} associated with
$(G,\theta_{1},\theta_{2})$.
We call $\{m(\mu,\epsilon) \mid \epsilon\in C_{k},\mu\in\Delta_{\epsilon}\}$
the \emph{multiplicities} of $\mathcal{S}$.
\end{dfn}

\subsubsection{Classification of rank one compact symmetric triads and their symmetric triads with multiplicities}

Let $G$ be a compact connected simple Lie group.
Let $(G,K_{1},K_{2})$
be a canonical compact symmetric triad with rank one.
We explicitly give the symmetric triad
$\mathcal{S}=(\Delta,\{\Delta_{\epsilon}\}_{\epsilon\in C_{k}})$
and its multiplicities $\{m(\mu,\epsilon) \mid \mu\in\Delta,\epsilon\in C_{k}\}$
associated with $(G,K_{1},K_{2})$,
where we put $k:=\mathrm{ord}(G,K_{1},K_{2})$.
By the classification of compact symmetric triads with rank one,
we have $k\in\{1,2,4\}$ (cf.~\cite[Table 4]{BI24}).
As in Table \ref{table:CST_rank1},
$\mathcal{S}$ is divided into eight types.

\begin{ex}
Let $\mathcal{S}=(\Delta,\{\Delta_{\epsilon}\}_{\epsilon\in C_{k}})$
be the symmetric triad associated with $(G,K_{1},K_{2})$.
In the case when $\Delta$ is of $A_{1}$-type,
we write $\Delta=\{\pm \mu\}$,
and in the case when $\Delta$ is of $BC_{1}$-type, we also write
$\Delta=\{\pm \mu,\pm 2\mu\}$.
For each $\epsilon\in C_{k}$,
the subset $\Delta_{\epsilon}$ defined by \eqref{eqn:dfn_delta_epsilon}
is described as follows:
\begin{enumerate}
\item[(1)] $\mathcal{S}$ is of
$(\mbox{III-}A_{1})$-type:
$\Delta=\Delta_{1}=\Delta_{-1}=\{\pm\mu\}$
($k=2$)
\item[(2)] $\mathcal{S}$ is of
$(\mbox{IV-}A_{1})$-type:
$\Delta=\Delta_{1}=\{\pm\mu\}$
($k=1$)
\item[(3)] $\mathcal{S}$ is of 
$(\mbox{I-}BC_{1}\mbox{-}A_{1})$-type:
$\Delta=\Delta_{1}=\{\pm \mu,\pm 2\mu\}$,
$\Delta_{-1}=\{\pm \mu\}$
($k=2$)
\item[(4)] $\mathcal{S}$ is of
$(\mbox{II-}BC_{1})$-type:
$\Delta=\Delta_{-1}=\{\pm \mu,\pm 2\mu\}$,
$\Delta_{1}=\{\pm \mu\}$
($k=2$)
\item[(5)] $\mathcal{S}$ is of
$(\mbox{III-}BC_{1})$-type:
$\Delta=\Delta_{1}=\Delta_{-1}=\{\pm\mu, \pm 2\mu\}$
($k=2$)
\item[(6)] $\mathcal{S}$ is of
$(\mbox{IV-}BC_{1})$-type:
$\Delta=\Delta_{1}=\{\pm\mu, \pm 2\mu\}$
($k=1$)
\item[(7)] $\mathcal{S}$ is of
$(\mbox{V-}BC_{1}\mbox{-}A)$-type:
\begin{equation}
\Delta=\Delta_{1}=\Delta_{-1}=\{\pm \mu,\pm 2\mu\},\quad
\Delta_{\sqrt{-1}}=\Delta_{-\sqrt{-1}}=\{\pm \mu\}
\quad
(k=4)
\end{equation}
\item[(8)] $\mathcal{S}$ is of
$(\mbox{V-}BC_{1}\mbox{-}D)$-type:
\begin{equation}
\Delta=\Delta_{-1}=\{\pm \mu,\pm 2\mu\},\quad
\Delta_{1}=\Delta_{\sqrt{-1}}=\Delta_{-\sqrt{-1}}=\{\pm \mu\}
\quad
(k=4)
\end{equation}
\end{enumerate}
In the case of $k=1$,
we have $\theta_{1}=\theta_{2}$.
Hence,
$\mathcal{S}$ is nothing but the restricted root system of the compact symmetric pair
$(G,\theta_{1})$, which was determined including its multiplicities
by Araki \cite{Araki}.
In the case of $k=2$,
the first author and Ikawa
\cite[Table 8]{BI25}
determined $\mathcal{S}$ including its multiplicities.
In the case of $k=4$,
the compact symmetric triad
$(G,\theta_{1},\theta_{2})$ is locally isomorphic to
either $(SU(2m),Sp(m),S(U(3)\times U(2m-3)))$ or $(SO(2m),SO(3)\times SO(2m-3),U(m))$.
By using explicit matrix realization of the compact symmetric triads,
one can determine $\mathcal{S}$ including its multiplicities, by direct computation.
Namely, we obtain:
\begin{itemize}
\item The symmetric triad
$\mathcal{S}$ for $(SU(2m),Sp(m),S(U(3)\times U(2m-3)))$ is of
(V-$BC_{1}$-$A$)-type.
Furthermore,
$\{m(\mu,\epsilon) \mid \mu\in\Delta,\epsilon\in C_{4}\}$
is given as follows:
\begin{equation}\label{eqn:mult_V-BC1-A}
\begin{cases}
(m(\mu,1),m(\mu,\sqrt{-1}),m(\mu,-1),m(\mu,-\sqrt{-1}))
=(4,4(m-3),4,4(m-3)),\\
(m(2\mu,1),m(2\mu,\sqrt{-1}),m(2\mu,-1),m(2\mu,-\sqrt{-1}))
=(1,0,3,0).
\end{cases}
\end{equation}
\item The symmetric triad
$\mathcal{S}$ for
$(SO(2m),SO(3)\times SO(2m-3),U(m))$
is of (V-$BC_{1}$-$D$)-type.
Furthermore,
$\{m(\mu,\epsilon) \mid \mu\in\Delta,\epsilon\in C_{4}\}$
is given as follows:
\begin{equation}\label{eqn:mult_V-BC1-D}
\begin{cases}
(m(\mu,1),m(\mu,\sqrt{-1}),m(\mu,-1),m(\mu,-\sqrt{-1}))
=(2,2(m-3),2,2(m-3)),\\
(m(2\mu,1),m(2\mu,\sqrt{-1}),m(2\mu,-1),m(2\mu,-\sqrt{-1}))
=(0,0,1,0).
\end{cases}
\end{equation}
\end{itemize}
\end{ex}

\subsubsection{Orbit space of cohomogeneity one Hermann actions}

Let $M$ be a principal orbit of the $K_{2}$-action on $G/K_{1}$,
and take $x=\pi_{1}(\exp_{G} Z)$ ($Z\in\mathfrak{b}$)
as a base point of $M$.
If we put $g=\exp_{G} Z$, then 
$T^{\perp}_{x}M=g_{*}\mathfrak{b}$ holds.
Let $\bs{\xi}$ denote the unit normal vector field on $M$
induced from a unit vector $\xi\in\mathfrak{b}$.
In particular, we get $\bs{\xi}_{x}=g_{*}\xi$.
By \eqref{eqn:p1_decomp2},
we have the following decomposition of $T_{x}M$:
\begin{equation}\label{eqn:orbit_pcurvature}
T_{x}M
=g_{*}(\mathfrak{p}_{1}(\mathfrak{b},0)\ominus\mathfrak{b})
\oplus
\bigoplus_{\epsilon\in C_{k}}\bigoplus_{\mu\in\Lambda^{R}_{\epsilon}}
g_{*}\mathfrak{p}_{1}(\mathfrak{b},\mu,\epsilon).
\end{equation}
By comparing this with the expression for the tangent space of a general $K_{2}$-orbit
given in \cite[p.~74]{Oh23}, we obtain the following relation:
\begin{equation}\label{eqn:principal_condition}
\mu(Z)+\varphi_{\epsilon}\not\in\pi\mathbb{Z}
\quad
(\epsilon\in C_{k},\,\mu\in\Lambda^{R}_{\epsilon}),
\end{equation}
where
$\varphi_{\epsilon}\in(-\pi/2,\pi/2]$ denotes
the constant
defined by
$\epsilon=e^{2\sqrt{-1}\varphi_{\epsilon}}$
for each $\epsilon\in C_{k}$.
The converse of \eqref{eqn:principal_condition}
is also true, namely,
it is shown that
any $K_{2}$-orbits satisfying 
\eqref{eqn:principal_condition}
is a principal orbit.

\begin{dfn}
Each connected component of the set of all elements $Z\in\mathfrak{b}$
satisfying \eqref{eqn:principal_condition}
is called a \emph{fundamental domain} of $\mathcal{S}$.
\end{dfn}

We can give
a fundamental domain of $\mathcal{S}$ which has the following expression:
\begin{equation}\label{eqn:dfn_fund_domain}
C:=\{t\cdot\xi\mid 0<t<l_{*}\},
\end{equation}
where $l_{*}$ is the positive real number which is give by
\begin{equation}
l_{*}=\min_{\epsilon\in C_{k},\,\mu\in\Lambda^{R}_{\epsilon}}\{t >0 \mid \sin(\mu(t\xi)+\varphi_{\epsilon})=0\}.
\end{equation}
Namely,
%
%
%
%
$\{K_{2}\cdot \pi_{1}(\exp_{G} Z)\mid Z\in\overline{C}\}$
gives the orbit foliation
(cf.~\cite[Theorem 1.6, (h)]{TeTh95}).
Its regular leaves consist of the orbits through points in the interior of
$\overline{C}$, i.e., the points in $C$.
On the other hand,
its singular leaves are the $K_{2}$-orbits through points on the boundary of
$\overline{C}$, which consists of two points.

\subsubsection{Principal curvature data for principal orbits}\label{sec:p-data}

Let $M$ be a principal orbit of the $K_{2}$-action,
and take $x=\pi_{1}(\exp_{G} Z)$
$(Z\in\mathfrak{b})$
as a base point of $M$.
We put $g=\exp_{G} Z$.
By \eqref{eqn:principal_condition},
we note that
$\sin(\mu(Z)+\varphi_{\epsilon})\neq 0$ holds
for all $\epsilon\in C_{k}$ and
$\mu\in\Lambda^{R}_{\epsilon}$.

\begin{pro}[{\cite[Corollary 4.28]{Ikawa11}, \cite[p. 79]{Oh23}
(cf.~\cite[Theorem 5.3]{GoTh07})}]\label{pro:orbit_pcurvature}
Let $A_{\bs{\xi}_{x}}$ denote the shape operator of $M$.
Then we have:
\begin{equation}
A_{\bs{\xi}_{x}}|_{g_{*}\mathfrak{p}_{1}(\mathfrak{b},\mu,\epsilon)}
=-\mu(\xi)\cot(\mu(Z)+\varphi_{\epsilon})\mathrm{id}_{g_{*}\mathfrak{p}_{1}(\mathfrak{b},\mu,\epsilon)},
\end{equation}
for each $\epsilon\in C_{k}$ and $\mu\in\Lambda^{R}_{\epsilon}$,
and $A_{\bs{\xi}_{x}}|_{g_{*}(\mathfrak{p}_{1}(\mathfrak{b},0)\ominus\mathfrak{b})}=0$.
In particular, the set of all eigenvalues of $A_{\bs{\xi}_{x}}$ is expressed as follows:
\begin{equation}\label{eqn:prin_M}
\begin{cases}
\{0\}\cup\{-\mu(\xi)\cot(\mu(Z)+\varphi_{\epsilon})\mid \epsilon\in C_{k},\,\mu\in\Lambda^{R}_{\epsilon}\} & (\mathfrak{p}_{1}(\mathfrak{b},0)\neq\mathfrak{b}),\\
\{-\mu(\xi)\cot(\mu(Z)+\varphi_{\epsilon})\mid \epsilon\in C_{k},\,\mu\in\Lambda^{R}_{\epsilon}\}
& (\mathfrak{p}_{1}(\mathfrak{b},0)=\mathfrak{b}).
\end{cases}
\end{equation}
\end{pro}

The condition $\mathfrak{p}_{1}(\mathfrak{b},0)=\mathfrak{b}$
is equivalent that $\mathfrak{b}$ is a maximal abelian subspace of $\mathfrak{p}_{1}$.
Hence, we can determine whether the latter condition holds or not
by using \cite[Table V]{Helgason} for $(G,K_{1},K_{2})$.
We can also determine the symmetric triad corresponding to $(G,K_{1},K_{2})$.
Hence the cardinality of the set \eqref{eqn:prin_M} is obtained.
Then it is shown that 
the number of distinct principal curvatures is equal to
$1,2,3,4,5,6$ or $7$ by case-by-case argument based on the classification.


\begin{notation}\label{notatoin:lambda_mepz}
Define $\lambda_{\mu,\epsilon,Z}\in\mathfrak{b}^{*}$ by
\begin{equation}
\lambda_{\mu,\epsilon,Z}(\xi):=
-\mu(\xi)\cot(\mu(Z)+\varphi_{\epsilon}).
\end{equation}
That is, for any $Y=t\cdot\xi\in \mathfrak{b}=\mathbb{R}\xi$,
we set
$\lambda_{\mu,\epsilon,Z}(Y):=t\lambda_{\mu,\epsilon,Z}(\xi)$.
\end{notation}

%
%

It follows from Proposition \ref{pro:orbit_pcurvature}
that the function $F_{\bs{\xi}_{x}}$ defined by
\eqref{eqn:Fxi_dfn} has the following expression:
\begin{align}
F_{\bs{\xi}_{x}}(t)
&=\prod_{\epsilon\in C_{k}}\prod_{\mu\in\Lambda^{R}_{\epsilon}}
 \left(\cos(t\mu(\xi))-\dfrac{\lambda_{\mu,\epsilon,Z}(\xi)}{\mu(\xi)}\sin(t\mu(\xi))\right)^{m(\mu,\epsilon)}\\
&=\prod_{\epsilon\in C_{k}}\prod_{\mu\in\Lambda^{R}_{\epsilon}}
\left\{
\dfrac{\sin(\mu(t\xi+Z)+\varphi_{\epsilon})}{\sin(\mu(Z)+\varphi_{\epsilon})}
\right\}^{m(\mu,\epsilon)}.\label{eqn:F_xi_triad}
\end{align}

\subsection{Determination of tangential focal data for principal orbits}

\subsubsection{Determination of $l_{\mu}$}

Let $\Delta$ denote the restricted root system of $(G,\theta_{1},\theta_{2})$.
Let $\xi$ be a unit vector in $\mathfrak{b}$
and we set $\Lambda^{R}:=\{\mu\in\Delta\mid \mu(\xi)>0\}$.
For each $\mu\in\Lambda^{R}$,
we write $l_{\mu}$ as the integer satisfying
$\mu(l \cdot \xi)=\pi \cdot l_{\mu}$
(cf.~Proposition \ref{pro:normal_circle_period}).

\begin{pro}\label{pro:lmu_determ}
Assume that $G/K_{1}$ is simply-connected.
Then we have the following assertions:
\begin{enumerate}
\item[(1)] $\Delta$ is of $A_{1}$-type:
If we put 
$\Lambda^{R}=\{\mu\}$,
then $l_{\mu}=2$ holds.
\item[(2)] $\Delta$ is of $BC_{1}$-type:
If we put 
$\Lambda^{R}=\{\mu,2\mu\}$,
then $l_{\mu}=1$ and $l_{2\mu}=2$ hold.
\end{enumerate}
\end{pro}

We prepare for the proof of Proposition \ref{pro:lmu_determ}.
Without loss of generalities
we may assume that $M$ contains the origin $eK_{1}$ of $G/K_{1}$.
As shown in the proof of Proposition \ref{pro:normal_circle_period},
we get
\begin{equation}\label{eqn:explxi_K}
\exp_{G}(l\cdot\xi)\in K_{1}.
\end{equation}
Let $\mathfrak{b}_{1}$ be a maximal abelian subspace of $\mathfrak{p}_{1}$
containing $\mathfrak{b}$,
and $\mathfrak{b}_{1}^{K_{1}}$ denote the lattice of $\mathfrak{b}_{1}$ defined by
\begin{equation}
\mathfrak{b}_{1}^{K_{1}}
:=\{Z\in\mathfrak{b}_{1}\mid \exp_{G}(Z)\in K_{1}\}.
\end{equation}
We recall another lattice of $\mathfrak{b}_{1}$
given by 
the restricted root system $\Phi_{1}$ of
$(G,K_{1})$ with respect to $\mathfrak{b}_{1}$.
Let $\mathfrak{b}_{1}^{\Phi_{1}}$ denote
the lattice of $\mathfrak{b}_{1}$
associated with $\Phi_{1}$ defined as follows:
\begin{equation}
\mathfrak{b}_{1}^{\Phi_{1}}:=\sum_{\alpha\in\Phi_{1}}\pi\mathbb{Z}\cdot H_{\alpha},
\quad
H_{\alpha}:=\dfrac{2}{\INN{\alpha}{\alpha}}\alpha^{\#}.
\end{equation}
Here, $\alpha^{\#}\in\mathfrak{b}_{1}$ is defined by
$\alpha(Z)=\INN{\alpha^{\#}}{Z}$
for all $Z\in\mathfrak{b}_{1}$.

\begin{pro}[{\cite[COROLLARY, p.~77]{LoosII}}]\label{pro:Loos_lattice}
We assume that
$G/K_{1}$ is simply-connected.
Then the following equality holds:
\begin{equation}
\mathfrak{b}_{1}^{K_{1}}=\mathfrak{b}_{1}^{\Phi_{1}}.
\end{equation}
\end{pro}

By \eqref{eqn:explxi_K} and Proposition \ref{pro:Loos_lattice},
we have $l\cdot\xi\in\mathfrak{b}_{1}^{\Phi_{1}}$.
Moreover, we note that the following equality holds:
\begin{equation}\label{eqn:l_minmal}
l=\min\{t>0\mid t\cdot\xi \in \mathfrak{b}_{1}^{\Phi_{1}}\}.
\end{equation}
There exist integers $m_{\alpha}$
$(\alpha\in\Phi_{1})$
such that
\begin{equation}\label{eqn:lxi_malphahalpha}
l\cdot\xi
=\sum_{\alpha\in\Phi_{1}}\pi m_{\alpha}\cdot H_{\alpha}.
\end{equation}
For $\mu\in\Lambda^{R}$,
there exists $\gamma\in\Phi_{1}$ satisfying
$\mu=\gamma|_{\mathfrak{b}}$.
Then the following equality holds:
\begin{equation}\label{eqn:mulxi_formula}
\mu(l\cdot\xi)
=\sum_{\alpha\in\Phi_{1}}\pi m_{\alpha}\cdot \gamma(H_{\alpha}).
\end{equation}
We determine the value of $\mu(l\cdot\xi)$ by using
\eqref{eqn:mulxi_formula} combined with \eqref{eqn:l_minmal}.

In order to obtain 
the integers
integers $m_{\alpha}$
$(\alpha\in\Phi_{1})$,
we determine the position of $\mathfrak{b}$ in $\mathfrak{b}_{1}$ concretely.
Our method of this determination
is based on the double Satake diagram $(S_{1},S_{2})$ of
$(G,\theta_{1},\theta_{2})$ (cf. Subsection \ref{sec:can_CST}).
Since $(G,\theta_{1},\theta_{2})$ is canonical,
without loss of generalities
we may assume that there exists a maximal abelian subalgebra $\mathfrak{t}$ of $\mathfrak{g}$
which is canonical with respect to $(G,\theta_{1},\theta_{2})$
and that
$\mathfrak{b}=\mathfrak{t}\cap(\mathfrak{p}_{1}\cap\mathfrak{p}_{2})$
and
$\mathfrak{b}_{1}=\mathfrak{t}\cap\mathfrak{p}_{1}$ hold
due to Theorem \ref{thm:BK24_thm}, (2).
We put $\sigma_{i}=(-\theta_{i})|_{\mathfrak{t}}$
($i=1,2$).
Let $\Pi$ be the fundamental system corresponding to the set of positive roots of $\Delta(\mathfrak{g},\mathfrak{t})$
for the ordering given in (2) of Definition \ref{dfn:canonical}.
Furthermore, let
$\mathrm{pr}_{1}:\mathfrak{t}\to\mathfrak{b}_{1}$
and
$\mathrm{pr}:\mathfrak{t}\to\mathfrak{b}$
be the orthogonal projections from $\mathfrak{t}$ onto
$\mathfrak{b}_{1}$ and $\mathfrak{b}$, respectively.
Then $\mathrm{pr}_{1}$ and $\mathrm{pr}$ are described as follows:
\begin{equation}\label{eqn:pr_pri_express}
\mathrm{pr}_{1}(Z)=\dfrac{1}{2}(Z+\sigma_{1}(Z)),\quad
\mathrm{pr}(Z)=\dfrac{1}{2k}
\sum_{j=0}^{k-1}\left\{
(\sigma_{1}\sigma_{2})^{j}(Z)+(\sigma_{1}\sigma_{2})^{j}\sigma_{1}(Z)
\right\},
\end{equation}
for all $Z\in\mathfrak{t}$.
Here, we note that $\mathfrak{t}=\mathrm{span}_{\mathbb{R}}\Pi$ holds.
For each $\alpha\in\Pi$, when $\sigma_{1}(\alpha)$ is expressed
as a linear combination of elements of $\Pi$, its coefficients can be read off
from the double Satake diagram $(S_{1},S_{2})$.
Therefore, by using \eqref{eqn:pr_pri_express}, we obtain the positions of
$\mathfrak{b}_{1}$ and $\mathfrak{b}$ in $\mathfrak{t}$.
From these, we obtain the position of $\mathfrak{b}$ in $\mathfrak{b}_{1}$.

With the above preparation, the proof of Proposition \ref{pro:lmu_determ}
is carried out based on the classification of $(G,\theta_{1},\theta_{2})$.

\begin{ex}[No.~10 in Table \ref{table:CST_rank1}]\label{ex:lmu_SOSOn-1OcOd}
Let $(G,K_{1},K_{2})=(SO(n),SO(n-1),SO(c)\times SO(n-c))$.
Here, $n\geq5$ and $2\leq c\leq n-2$.
Let $\Delta=\{\pm \mu\}$ be the restricted root system of this compact symmetric triad.
Then we have $\mu(l\cdot \xi)=2\pi$.
\end{ex}

\begin{proof}
We divide the proof according to whether $n$ is even or odd.
First, consider the case of $n=2m$.
Table \ref{table:SOSOn-1SOcSOd} gives the double Satake diagram of
$(SO(n),SO(n-1),SO(c)\times SO(n-c))$.
In this case, we have
$\mu=\alpha_{1}|_{\mathfrak{b}}$.
Since $\mathfrak{b}_{1}=\mathfrak{b}$ holds,
we obtain
$\Phi_{1}=\{\pm \mu\}$.
Then there exists an integer $q$ such that
$l\cdot\xi=\pi q H_{\mu}$.
Hence
$\mu(l\cdot\xi)=2\pi q$.
By \eqref{eqn:l_minmal}, we have $q=1$.
Therefore,
$\mu(l\cdot\xi)=2\pi$ holds.
In the case where $n$ is odd,
a similar argument shows $\mu(l\cdot\xi)=2\pi$.
\end{proof}

\begin{table}[!!ht]
\renewcommand{\arraystretch}{1.1}
\setlength{\tabcolsep}{10pt}
\centering
\caption{Double Satake diagram $(S_{1},S_{2})$ of $(SO(n),SO(n-1),SO(c)\times SO(n-c))$
with $n\geq 8$,
and $2\leq c\leq m-2$
}\label{table:SOSOn-1SOcSOd}
\begin{tabular}{cc}
\toprule
Satake diagram $S_{1}$ & Satake diagram $S_{2}$\\
\toprule
\begin{xy}
\ar@{-}(0,0)*++!D{\alpha_{1}}*{\circ};(10,0)*++!D{\alpha_{2}}*{\bullet}="a2"
\ar@{-}"a2";(15,0)*{}
\ar@{.}(15,0)*{};(20,0)*{}
\ar@{-}(20,0)*{};(25,0)*++!D{\alpha_{m-3}}*{\bullet}="a2m-3"
\ar@{-}"a2m-3";(35,0)*++!L{\alpha_{m-2}}*{\bullet}="a2m-2"
\ar@{-}"a2m-2";(40,8)*++!L{\alpha_{m-1}}*{\bullet}
\ar@{-}"a2m-2";(40,-8)*++!L{\alpha_{m}}*{\bullet}
\end{xy}
&
\begin{xy}
\ar@{-}(10,0)*++!D{\alpha_{1}}*{\circ};(15,0)*{}
\ar@{.}(15,0)*{};(20,0)*{}
\ar@{-}(20,0)*{};(25,0)*++!D{\alpha_{c}}*{\circ}="ac"
\ar@{-}"ac";(35,0)*++!D{\alpha_{c+1}}*{\bullet}="ac+1"
\ar@{-}"ac+1";(40,0)*{}
\ar@{.}(40,0)*{};(45,0)*{}
\ar@{-}(45,0)*{};(50,0)*++!L{\alpha_{m-2}}*{\bullet}="a2m-2"
\ar@{-}"a2m-2";(55,8)*++!L{\alpha_{m-1}}*{\bullet}
\ar@{-}"a2m-2";(55,-8)*++!L{\alpha_{m}}*{\bullet}
\end{xy}\\
\bottomrule
\end{tabular}
\end{table}

\begin{ex}[No.~4 in Table \ref{table:CST_rank1}]\label{ex:lmu_SUSUn-1SUcUd}
Let $(G,\theta_{1},\theta_{2})=
(SU(n),S(U(1)\times U(n-1)),S(U(c)\times U(n-c)))$.
Here, $2\leq c\leq n-2$.
Let $\Delta=\{\pm \mu, \pm 2\mu\}$ be the restricted root system of this compact symmetric triad.
Then we have $\mu(l\cdot \xi)=\pi$ and $2\mu(l\cdot \xi)=2\pi$.
\end{ex}

\begin{proof}
Table \ref{table:SUSUn-1SUcUd} gives the double Satake diagram of
$(SU(n),S(U(1)\times U(n-1)),S(U(c)\times U(n-c)))$.
In this case, we have
$\mu=\alpha_{1}|_{\mathfrak{b}}$
and
$2\mu=(\alpha_{1}+\dotsc+\alpha_{n-1})|_{\mathfrak{b}}$.
Since $\mathfrak{b}_{1}=\mathfrak{b}$ holds,
we obtain
$\Phi_{1}=\{\pm\mu,\pm2\mu\}$.
It follows from $H_{\mu}=2H_{2\mu}$
that
$\mathfrak{b}_{1}^{\Phi_{1}}=\pi\mathbb{Z}H_{2\mu}$ holds.
Hence there exists an integer $q$ satisfying
$l\cdot\xi=\pi q H_{2\mu}$.
Then we get
$\mu(l\cdot\xi)=\pi q$.
By \eqref{eqn:l_minmal}, we have $q=1$,
namely,
$\mu(l\cdot\xi)=\pi$ holds.
From which, we immediately obtain
$2\mu(l\cdot\xi)=2\pi$.
\end{proof}

\begin{table}[!!ht]
\renewcommand{\arraystretch}{1.1}
\setlength{\tabcolsep}{10pt}
\centering
\caption{Double Satake diagram $(S_{1},S_{2})$ of $(SU(n),S(U(1)\times U(n-1)),S(U(c)\times U(n-c)))$
with $n=c+d$, $1<c\leq d$
}\label{table:SUSUn-1SUcUd}
\begin{tabular}{cc}
\toprule
Satake diagram $S_{1}$ & Satake diagram $S_{2}$\\
\toprule
\begin{xy}
\ar@{-}(0,0)*++!D{\alpha_{1}}*{\circ}="a1";(10,0)*++!D{\alpha_{2}}*{\bullet}="a2"
\ar@{-}"a2";(15,0)*{}
\ar@{.}(15,0)*{};(20,0)*{}
\ar@{-}(20,0)*{};(25,0)*++!D{\alpha_{n-2}}*{\bullet}="a2m-3"
\ar@{-}"a2m-3";(35,0)*++!D{\alpha_{n-1}}*{\circ}="an-1"
\ar@/_6mm/@{<->} "a1";"an-1"
\end{xy}
&
\begin{xy}
\ar@{-}(-15,0)*++!D{\alpha_{1}}*{\circ}="a1";(-10,0)*{}
\ar@{.}(-10,0)*{};(-5,0)*{}
\ar@{-}(0,0)*++!D{\alpha_{c}}*{\circ}="ac";(10,0)*++!D{\alpha_{c+1}}*{\bullet}="a2"
\ar@{-}(-5,0)*{};"ac"
\ar@{-}"a2";(15,0)*{}
\ar@{.}(15,0)*{};(20,0)*{}
\ar@{-}(20,0)*{};(25,0)*++!D{\alpha_{n-c-1}}*{\bullet}="a2m-3"
\ar@{-}"a2m-3";(35,0)*++!D{\alpha_{n-c}}*{\circ}="an-c"
\ar@{-}"an-c";(40,0)*{}
\ar@{.}(40,0)*{};(45,0)*{}
\ar@{-}(45,0)*{};(50,0)*++!D{\alpha_{n-1}}*{\circ}="an-1"
\ar@/_6mm/@{<->} "a1";"an-1"
\ar@/_4mm/@{<->} "ac";"an-c"
\ar@{}(0,-8)
\end{xy}
\\
\bottomrule
\end{tabular}
\end{table}


If $\Delta$ is of type $A_{1}$,
the same argument as in Example \ref{ex:lmu_SOSOn-1OcOd}
shows that $l_{\mu}=2$.
If $\Delta$ is of type $BC_{1}$,
the same argument as in Example \ref{ex:lmu_SUSUn-1SUcUd}
shows that
$l_{\mu}=1$
and
$l_{2\mu}=2$.
This completes the proof of Proposition \ref{pro:lmu_determ}.

\subsubsection{Determination of tangential focal data for principal orbits}

Let $G$ be a compact connected simple Lie group.
Let $(G,K_{1},K_{2})$ be a canonical compact symmetric triad with rank one.
Let $\mathcal{S}:=(\Delta,\{\Delta_{\epsilon}\}_{\epsilon\in C_{k}})$
denote its symmetric triad
and $\{m(\mu,\epsilon)\mid \epsilon\in C_{k},\mu\in\Delta_{\epsilon}\}$ denote
its multiplicities,
where we put $k:=\mathrm{ord}(G,K_{1},K_{2})$ for the simplicity.
Let $\xi\in\mathfrak{b}$ be a unit vector.
Let $C$ be the fundamental domain of $\mathcal{S}$ given by
\eqref{eqn:dfn_fund_domain}.
If
we put $x:=\pi_{1}(\exp_{G} Z)$
$(Z\in C)$,
then
$M:=K_{2}\cdot x$
is a principal orbit.
We write the tangential focal data of $M$ at $x$ as
$\Gamma^{F}(M,x)=\{\theta_{F},l/(2g_{F}),\{m_{1}^{F},m_{2}^{F}\};2g_{F}\}$.
We have $C=\{t\cdot\xi \mid 0<t<l/(2g_{F})\}$.
The length $l$ of the normal circle of $M$ at $x$
is given by using \eqref{eqn:l_formula_tildem}, \eqref{eqn:tildem_m}
and Proposition \ref{pro:lmu_determ}.

\begin{lem}\label{lem:theta_F_expression}
In the case when $\Delta$ is of type $A_{1}$, we write
$\Lambda^{R}=\{\mu\}$,
and when $\Delta$ is of type $BC_{1}$, we write
$\Lambda^{R}=\{\mu,2\mu\}$.
Then we have
$\theta_{F}=(l/(2g_{F}))-(\mu(Z)/\mu(\xi))$
and 
$\bar{\theta}_{F}:=\pi(2g_{F}/l)\theta_{F}$
has the following expression:
\begin{equation}\label{eqn:theta_F_gen_expression}
\bar{\theta}_{F}=\pi-\dfrac{2g_{F}}{l_{\mu}}\mu(Z).
\end{equation}
\end{lem}

\begin{proof}
This lemma is immediately obtained from
$\theta_{F}\xi+Z=(l/(2g_{F}))\xi$
and $\mu(l\cdot\xi)=l_{\mu}\pi$.
\end{proof}

\begin{thm}\label{thm:determ_l}
Retain the above setting.
Then
$\{m_{1}^{F},m_{2}^{F}\}$ and $g_{F}$
are given in Table \ref{table:cohom1Hermann_data}.
In particular,
$\{m_{1}^{F},m_{2}^{F}\}$
is described by using 
$\{m(\mu,\epsilon)\mid \epsilon\in C_{k},\mu\in\Delta_{\epsilon}\}$,
where $k$ denotes the order of $(G,K_{1},K_{2})$.
\end{thm}

\begin{table}[H]
\renewcommand{\arraystretch}{1.1}
\setlength{\tabcolsep}{10pt}
\centering
\caption{Data of
principal orbits of cohomogeneity one Hermann actions:
$g_{F}$, $\{m^{F}_{1},m^{F}_{2}\}$
}\label{table:cohom1Hermann_data}
\begin{tabular}{ccc}
\toprule
\begingroup
\begin{tabular}{c}
Type of\\
symmetric triad
\end{tabular}
\endgroup
& $g_{F}$ & $m^{F}_{1},m^{F}_{2}$  \\
\toprule
$(\mbox{III-}A_{1})$& $2$  & $m(\mu,-1)$, $m(\mu,1)$  \\\midrule
$(\mbox{IV-}A_{1})$ & $1$  & $m(\mu,1)$ \\\midrule
$(\mbox{I-}BC_{1}\mbox{-}A_{1})$&  $1$  & $m(\mu,-1)+m(2\mu,1)$, $m(\mu,1)+m(2\mu,1)$   \\\midrule
$(\mbox{II-}BC_{1})$ & $2$  & $m(2\mu,-1)$, $m(\mu,1)$  \\\midrule
$(\mbox{III-}BC_{1})$ & $2$ & $m(2\mu,-1)$, $m(\mu,1)+m(2\mu,1)$  \\\midrule
$(\mbox{IV-}BC_{1})$ & $1$  & $m(2\mu,1)$, $m(\mu,1)+m(2\mu,1)$   \\\midrule
$(\mbox{V-}BC_{1}\mbox{-}A)$& $2$ & $m(\mu,\sqrt{-1})+m(2\mu,-1)$, $m(\mu,1)+m(2\mu,1)$   \\\midrule
$(\mbox{V-}BC_{1}\mbox{-}D)$ & $2$ & $m(\mu,\sqrt{-1})+m(2\mu,-1)$, $m(\mu,1)$  \\
\bottomrule
\end{tabular}
\end{table}

The proof of Theorem \ref{thm:determ_l} is
carried out case by case according to the type of the symmetric triad $\mathcal{S}$.
Here, the choice of $Z$ is normalized so that $Z\in C$.
Hence, we note that the multiplicities of focal normals are obtained as the ordered pair
$(m_{1}^{F},m_{2}^{F})$ instead of the unordered pair $\{m_{1}^{F},m_{2}^{F}\}$.
For $j\in\{1,\dotsc,2g_{F}\}$, we put $t_{j}:=\theta_{F}+(j-1)(l/(2g_{F}))$.
In the case when $\Delta$ is of type $A_{1}$, we write
$\Lambda^{R}=\{\mu\}$,
and when $\Delta$ is of type $BC_{1}$, we write
$\Lambda^{R}=\{\mu,2\mu\}$.

\begin{lem}\label{lem:IIIA1_fdata}
In the case when $\mathcal{S}$ is of 
($\mbox{III-}A_{1}$)-type,
we have
$g_{F}=2$ and the following equality holds:
\begin{equation}\label{eqn:IIIA1_nullsp}
\mathrm{Ker}(d\eta)_{t\cdot\bs{\xi}}=\begin{cases}
g_{*}\mathfrak{p}_{1}(\mathfrak{b},\mu,-1) & (t=t_{1},t_{3}),\\
g_{*}\mathfrak{p}_{1}(\mathfrak{b},\mu,1) & (t=t_{2},t_{4}).
\end{cases}
\end{equation}
In particular,
$(m_{1}^{F},m_{2}^{F})=(m(\mu,-1),m(\mu,1))$ holds.
\end{lem}

\begin{proof}
We can write \eqref{eqn:principal_condition} explicitly by using
$\Delta=\Delta_{1}=\Delta_{-1}=\{\pm\mu\}$.
Then we get $\mu(l_{*}\xi)=\pi/2$.
It follows from \eqref{eqn:F_xi_triad}
that $F_{\bs{\xi}}$ is rewritten as follows:
\begin{equation}\label{eqn:IIIA1_Fxi}
F_{\bs{\xi}}(t)
=\left\{
\dfrac{\sin(\mu(t{\xi}+Z))}{\sin(\mu(Z))}
\right\}^{m(\mu,1)}
\left\{
\dfrac{\cos(\mu(t{\xi}+Z))}{\cos(\mu(Z))}
\right\}^{m(\mu,-1)}.
\end{equation}
By Proposition \ref{pro:lmu_determ}, we have $l_{\mu}=2$.
In the range $0<t<l$, the equation $F_{\bs{\xi}}(t)=0$ has four solutions.
Thus
$g_{F}=2$.
For $j\in\{1,\dotsc,4\}$, it follows from
\eqref{eqn:IIIA1_Fxi} that
$\mathrm{Ker}(d\eta)_{t_{j}\cdot\bs{\xi}}$
is given by
\eqref{eqn:IIIA1_nullsp}.
Hence we obtain
$m_{t_{1}\cdot\bs{\xi}}^{F}=m(\mu,-1)$
and
$m_{t_{2}\cdot\bs{\xi}}^{F}=m(\mu,1)$.
\end{proof}

A similar argument shows the following lemmas.
We omit their proofs.

\begin{lem}
In the case when $\mathcal{S}$ is of
$(\mbox{IV-}A_{1})$-type,
we have
$g_{F}=1$
and the following equality holds:
\begin{equation}\label{eqn:IVA1_nullsp}
\mathrm{Ker}(d\eta)_{t_{1}\cdot\bs{\xi}}=
\mathrm{Ker}(d\eta)_{t_{2}\cdot\bs{\xi}}=
g_{*}\mathfrak{p}_{1}(\mathfrak{b},\mu,1).
\end{equation}
In particular,
$m_{1}^{F}=m_{2}^{F}=m(\mu,1)$ holds.
\end{lem}

\begin{lem}\label{lem:focal_I-BC1-A1}
In the case when $\mathcal{S}$ is of
$(\mbox{I-}BC_{1}\mbox{-}A_{1})$-type,
we have
$g_{F}=1$
and the following equality holds:
\begin{equation}\label{eqn:IBC1_nullsp}
\mathrm{Ker}(d\eta)_{t\cdot\bs{\xi}}=\begin{cases}
g_{*}(\mathfrak{p}_{1}(\mathfrak{b},\mu,-1)\oplus\mathfrak{p}_{1}(\mathfrak{b},2\mu,1))
 & (t=t_{1}),\\
g_{*}(\mathfrak{p}_{1}(\mathfrak{b},\mu,1)\oplus\mathfrak{p}_{1}(\mathfrak{b},2\mu,1)) & (t=t_{2}).
\end{cases}
\end{equation}
In particular,
$(m_{1}^{F},m_{2}^{F})=(m(\mu,-1)+m(2\mu,1),m(\mu,1)+m(2\mu,1))$ holds.
\end{lem}

\begin{lem}
In the case when $\mathcal{S}$ is of
$(\mbox{II-}BC_{1})$-type,
we have
$g_{F}=2$
and the following equality holds:
\begin{equation}\label{eqn:IIBC1_nullsp}
\mathrm{Ker}(d\eta)_{t\cdot\bs{\xi}}=\begin{cases}
g_{*}\mathfrak{p}_{1}(\mathfrak{b},2\mu,-1) & (t=t_{1},t_{3}),\\
g_{*}\mathfrak{p}_{1}(\mathfrak{b},\mu,-1) & (t=t_{2}),\\
g_{*}\mathfrak{p}_{1}(\mathfrak{b},\mu,1) & (t=t_{4}).
\end{cases}
\end{equation}
In particular,
$(m_{1}^{F},m_{2}^{F})=(m(2\mu,-1),m(\mu,1))
=(m(2\mu,-1),m(\mu,-1))$ holds.
\end{lem}

\begin{lem}
In the case when $\mathcal{S}$ is of
$(\mbox{III-}BC_{1})$,
we have
$g_{F}=2$
and the following equality holds:
\begin{equation}
\mathrm{Ker}(d\eta)_{t\cdot\bs{\xi}}=\begin{cases}
g_{*}\mathfrak{p}_{1}(\mathfrak{b},2\mu,-1) & (t=t_{1},t_{3}),\\
g_{*}(\mathfrak{p}_{1}(\mathfrak{b},\mu,-1)\oplus\mathfrak{p}_{1}(\mathfrak{b},2\mu,1)) & (t=t_{2}),\\
g_{*}(\mathfrak{p}_{1}(\mathfrak{b},\mu,1)\oplus\mathfrak{p}_{1}(\mathfrak{b},2\mu,1)) & (t=t_{4}).
\end{cases}
\end{equation}
In particular,
$(m_{1}^{F},m_{2}^{F})=(m(2\mu,-1),m(\mu,1)+m(2\mu,1))=(m(2\mu,-1),m(\mu,-1)+m(2\mu,1))$
holds.
\end{lem}

\begin{lem}
In the case when $\mathcal{S}$ is of
$(\mbox{IV-}BC_{1})$,
we have
$g_{F}=1$
and the following equality holds:
\begin{equation}\label{eqn:IVBC1_nullsp}
\mathrm{Ker}(d\eta)_{t\cdot\bs{\xi}}=\begin{cases}
g_{*}\mathfrak{p}_{1}(\mathfrak{b},2\mu,1) & (t=t_{1}),\\
g_{*}(\mathfrak{p}_{1}(\mathfrak{b},\mu,1)\oplus\mathfrak{p}_{1}(\mathfrak{b},2\mu,1)) & (t=t_{2}).
\end{cases}
\end{equation}
In particular,
$(m_{1}^{F},m_{2}^{F})=(m(2\mu,1),m(\mu,1)+m(2\mu,1))$ holds.
\end{lem}

\begin{lem}
In the case when $\mathcal{S}$ is of
$(\mbox{V-}BC_{1}\mbox{-}A)$-type,
we have
$g_{F}=2$
and the following equality holds:
\begin{equation}\label{eqn:VABC1_nullsp}
\mathrm{Ker}(d\eta)_{t\cdot\bs{\xi}}=\begin{cases}
g_{*}(\mathfrak{p}_{1}(\mathfrak{b},\mu,-\sqrt{-1})\oplus\mathfrak{p}_{1}(\mathfrak{b},2\mu,-1)) & (t=t_{1}),\\
g_{*}(\mathfrak{p}_{1}(\mathfrak{b},\mu,-1)\oplus\mathfrak{p}_{1}(\mathfrak{b},2\mu,1)) & (t=t_{2}),\\
g_{*}(\mathfrak{p}_{1}(\mathfrak{b},\mu,\sqrt{-1})\oplus\mathfrak{p}_{1}(\mathfrak{b},2\mu,-1)) & (t=t_{3}),\\
g_{*}(\mathfrak{p}_{1}(\mathfrak{b},\mu,1)\oplus\mathfrak{p}_{1}(\mathfrak{b},2\mu,1)) & (t=t_{4}).
\end{cases}
\end{equation}
In particular, we have
\begin{align}
&m_{1}^{F}
=m(\mu,-\sqrt{-1})+m(2\mu,-1)=m(\mu,\sqrt{-1})+m(2\mu,-1),\\
&m_{2}^{F}
=m(\mu,-1)+m(2\mu,1)=m(\mu,1)+m(2\mu,1).
\end{align}
\end{lem}

\begin{lem}\label{lem:VDA1_fdata}
In the case when $\mathcal{S}$ is of
$(\mbox{V-}BC_{1}\mbox{-}D)$-type,
we have
$g_{F}=2$
and the following equality holds:
\begin{equation}\label{eqn:VDBC1_nullsp}
\mathrm{Ker}(d\eta)_{t\cdot\bs{\xi}}=\begin{cases}
g_{*}(\mathfrak{p}_{1}(\mathfrak{b},\mu,-\sqrt{-1})\oplus\mathfrak{p}_{1}(\mathfrak{b},2\mu,-1)) & (t=t_{1}),\\
g_{*}\mathfrak{p}_{1}(\mathfrak{b},\mu,-1) & (t=t_{2}),\\
g_{*}(\mathfrak{p}_{1}(\mathfrak{b},\mu,\sqrt{-1})\oplus\mathfrak{p}_{1}(\mathfrak{b},2\mu,-1)) & (t=t_{3}),\\
g_{*}\mathfrak{p}_{1}(\mathfrak{b},\mu,1) & (t=t_{4})
\end{cases}
\end{equation}
In particular, we have
\begin{align}
&m_{1}^{F}
=m(\mu,-\sqrt{-1})+m(2\mu,-1)=m(\mu,\sqrt{-1})+m(2\mu,-1),\\
&m_{2}^{F}
=m(\mu,-1)=m(\mu,1).
\end{align}
\end{lem}

\begin{proof}[Proof of Theorem \ref{thm:determ_l}]
Table \ref{table:cohom1Hermann_data}
is obtained from Lemmas \ref{lem:IIIA1_fdata}--\ref{lem:VDA1_fdata}.
\end{proof}

\begin{thm}\label{thm:mainthm1}
Retain the above setting.
Then, we have:
\begin{enumerate}
\item[(1)] The data $m_{1}^{F}$,
$m_{2}^{F}$,
$g_{F}$,
$l$, $\delta_{F}$
and $\epsilon_{F}$
are as in Table \ref{table:CST_focal}.
\item[(2)] Let $c$ be the constant defined in \eqref{eqn:barthetaF_determ}.
Then $\bar{\theta}_{F}=\pi-c\,\mu(Z)$ holds.
Furthermore, we have $0<\mu(Z)<\pi/c$.
\end{enumerate}
\end{thm}

\begin{table}[h!!]
\centering
\footnotesize
\renewcommand{\arraystretch}{1.1}
\setlength{\tabcolsep}{8pt}
\caption{Data associated with the tangential focal data for principal orbits of cohomogeneity one Hermann actions}\label{table:CST_focal}
\begin{tabular}{ccccccc}
\toprule
No. &
$m_{1}^{F},m_{2}^{F}$ &
$g_{F}$ &
$l$ &
$\delta_{F}$ &
$\epsilon_{F}$ & Remark \\
\toprule
1 &
$1, 1$ &
$1$ &
$\dfrac{2\,\pi}{\sqrt{\kappa}}$ &
$0$ &
$\dfrac{1}{2\,\kappa}$ &
\\\midrule
2 &
$1, n-2$ &
$2$ &
$\dfrac{\sqrt{2n}\,\pi}{\sqrt{\kappa}}$ &
$-\dfrac{n-3}{n-1}$ &
$\dfrac{n}{8(n-1)\,\kappa}$ &
$n\geq3$
\\\midrule
3 &
$1,2n-3$ &
$1$ &
$\dfrac{\sqrt{2n}\,\pi}{\sqrt{\kappa}}$ &
$-\dfrac{n-2}{n-1}$ &
$\dfrac{n}{4(n-1)\,\kappa}$ &
$n\geq2$
\\\midrule
4 &
$2c-1, 2n-2c-1$ & 
$1$ &
$\dfrac{\sqrt{2n}\,\pi}{\sqrt{\kappa}}$ &
$-\dfrac{n-2c}{n-1}$ &
$\dfrac{n}{4(n-1)\,\kappa}$ &
\begingroup
\renewcommand{\arraystretch}{0.9}
\begin{tabular}{c}
$n\geq 4$,\\
$2\leq c\leq n-2$
\end{tabular}
\endgroup
\\\midrule
5 &
$1, 3$ &
$2$ &
$\dfrac{4\,\pi}{\sqrt{\kappa}}$ &
$-\dfrac{1}{2}$ &
$\dfrac{1}{4\,\kappa}$ &
\\\midrule
6 &
$1,4m-5$ &
$2$ &
$\dfrac{2\sqrt{2m}\,\pi}{\sqrt{\kappa}}$ & 
$-\dfrac{2m-3}{2m-2}$ &
$\dfrac{m}{8(m-1)\,\kappa}$ &
$m \geq 3$
\\\midrule
7 &
$3,5$ &
$2$ &
$\dfrac{2\sqrt{6}\,\pi}{\sqrt{\kappa}}$ &
$-\dfrac{1}{4}$ &
$\dfrac{3}{16\,\kappa}$ &
\\\midrule
8 &
$4m-9, 5$ &
$2$ &
$\dfrac{2\sqrt{2m}\,\pi}{\sqrt{\kappa}}$ &
$\dfrac{2m-7}{2m-2}$ &
$\dfrac{m}{8(m-1)\,\kappa}$ &
$m\geq 4$
\\\midrule
9 &
$n-2,n-2$ &
$1$ &
$\dfrac{2\sqrt{n-2}\,\pi}{\sqrt{\kappa}}$ &
$0$ &
$\dfrac{1}{2\,\kappa}$ &
$n\geq 5$
\\
\bottomrule
\end{tabular}
\bigskip
\end{table}

\begin{table}[h!!]
\centering
\small
\renewcommand{\arraystretch}{1.1}
\caption*{\textsc{Table} \ref{table:CST_focal}. (continued)}
\begin{tabular}{ccccccc}
\toprule
No. &
$m_{1}^{F},m_{2}^{F}$ &
$g_{F}$ &
$l$ &
$\delta_{F}$ &
$\epsilon_{F}$ & Remark \\
\toprule
10 &
$c-1, n-c-1$ &
$2$ &
$\dfrac{2\sqrt{n-2}\,\pi}{\sqrt{\kappa}}$ &
$-\dfrac{n-2c}{n-2}$ &
$\dfrac{1}{4\,\kappa}$ &
\begingroup
\renewcommand{\arraystretch}{0.9}
\begin{tabular}{c}
$n\geq 5$,\\
$2\leq c\leq n-2$
\end{tabular}
\endgroup
\\\midrule
11 &
$1, 2$ &
$2$ &
$\dfrac{2\sqrt{3}\,\pi}{\sqrt{\kappa}}$ &
$-\dfrac{1}{3}$ &
$\dfrac{1}{4\,\kappa}$ &
\\\midrule
12 &
$2,2n-3$ &
$2$ &
$\dfrac{2\sqrt{n+1}\,\pi}{\sqrt{\kappa}}$ &
$-\dfrac{2n-5}{2n-1}$ &
$\dfrac{n+1}{4(2n-1)\,\kappa}$ &
$n\geq 3$
\\\midrule
13 &
$3,4n-5$ &
$1$ &
$\dfrac{2\sqrt{n+1}\,\pi}{\sqrt{\kappa}}$ &
$-\dfrac{2(n-2)}{2n-1}$ &
$\dfrac{n+1}{2(2n-1)\,\kappa}$ &
$n\geq 2$
\\\midrule
14 &
$4c-1, 4n-4c-1$ &
$1$ &
$\dfrac{2\sqrt{n+1}\,\pi}{\sqrt{\kappa}}$ &
$-\dfrac{2(n-2c)}{2n-1}$ &
$\dfrac{n+1}{2(2n-1)\,\kappa}$ &
\begingroup
\renewcommand{\arraystretch}{0.9}
\begin{tabular}{c}
$n\geq 4$,\\
$2\leq c\leq n-2$
\end{tabular}
\endgroup
\\\midrule
15 &
$2m-3,2m-3$ &
$1$ &
$\dfrac{2\sqrt{m-1}\,\pi}{\sqrt{\kappa}}$ &
$0$ &
$\dfrac{m-1}{2(2m-3)\,\kappa}$ &
$m\geq 3$
\\\midrule
16 &
$1,2$ &
$2$ &
$\dfrac{2\sqrt{2}\,\pi}{\sqrt{\kappa}}$ &
$-\dfrac{1}{3}$ &
$\dfrac{1}{6\,\kappa}$ &
\\\midrule
17 &
$2m-5, 2$ &
$2$ &
$\dfrac{2\sqrt{m-1}\,\pi}{\sqrt{\kappa}}$ &
$\dfrac{2m-7}{2m-3}$ &
$\dfrac{m-1}{4(2m-3)\,\kappa}$ &
$m\geq 4$
\\\midrule
18 &
$5,11$ &
$2$ &
$\dfrac{4\sqrt{3}\,\pi}{\sqrt{\kappa}}$ &
$-\dfrac{3}{8}$ &
$\dfrac{3}{16\,\kappa}$ &
\\\midrule
19 &
$1,15$ &
$2$ &
$\dfrac{4\sqrt{3}\,\pi}{\sqrt{\kappa}}$ &
$-\dfrac{7}{8}$ &
$\dfrac{3}{16\,\kappa}$ &
\\\midrule
20 &
$4,7$ &
$2$ &
$\dfrac{6\,\pi}{\sqrt{\kappa}}$ &
$-\dfrac{3}{11}$ &
$\dfrac{9}{44\,\kappa}$ &
\\\midrule
21 &
$7,15$ &
$1$ &
$\dfrac{6\,\pi}{\sqrt{\kappa}}$ &
$-\dfrac{4}{11}$ &
$\dfrac{9}{22\,\kappa}$ & \\
\bottomrule
\end{tabular}
\end{table}

\begin{proof}[Proof of Theorem \ref{thm:mainthm1}]
(1) The assertion is proved by carrying out computations similar to those in the above example,
based on the classification (see Table \ref{table:CST_rank1}).
(2) The constant $c$ defined in \eqref{eqn:barthetaF_determ}
is equal to $2g_{F}/l_{\mu}$.
By using Table \ref{table:cohom1Hermann_data}
and Proposition \ref{pro:lmu_determ},
the value of $2g_{F}/l_{\mu}$
is determined
for each symmetric triad corresponding to $(G,K_{1},K_{2})$.
By combining this with \eqref{eqn:theta_F_gen_expression},
we obtain the relation in (2).
\end{proof}

\begin{rem}
(1) The number in the column labeled ``No.'' in Table \ref{table:CST_focal}
refers to the compact symmetric triad with the corresponding number in Table \ref{table:CST_rank1}.
(2) The sign of the constant $\delta_{F}$ depends on the choice of a fundamental domain for the Hermann action.
The fundamental domain adopted in this paper has
the form given in \eqref{eqn:dfn_fund_domain}.
\end{rem}

\section{Cohomogeneity one Hermann actions and  backward mean curvature flows}\label{sec:MCF_Hermann}

Let $G$ be a compact connected simple Lie group.
Let $(G,K_{1},K_{2})$ be a rank one canonical compact symmetric triad
such that $G/K_{i}$ ($i=1,2$) is simply-connected.
Let $\mathcal{S}:=(\Delta,\{\Delta_{\epsilon}\}_{\epsilon\in C_{k}})$
denote its symmetric triad
and $\{m(\mu,\epsilon)\mid \mu\in\Delta,\epsilon\in C_{k}\}$ denote
its multiplicities,
where $k$ denotes the order of $(G,K_{1},K_{2})$.
Let $\xi\in\mathfrak{b}$ be a unit vector
and $C$ be 
the fundamental domain of $\mathcal{S}$
given in \eqref{eqn:dfn_fund_domain}.
For $Z\in C$,
let $\bs{H}^{Z}$ denote the mean curvature vector field
of the $K_{2}$-orbit $M_{z}:=K_{2}\cdot z$
($z:=\pi_{1}(\exp_{G}(Z))$).
Following to \cite[p. 106]{Koike11},
we define the vector field $\bs{X}$ on $C$ by
\begin{equation}\label{eqn:MCF_ODE_Hermann}
\bs{X}_{Z}:=\bs{H}^{Z}_{z}
\quad
(Z\in C),
\end{equation}
where we have used the identification of $T^{\perp}_{z}M_{z}$ and $\mathfrak{b}$.
It follows from \cite[COROLLARY 1]{Oh23} that
$\bs{H}^{Z}_{z}$ has the following expression:
\begin{equation}
\bs{H}^{Z}_{z}
=-\sum_{\epsilon\in C_{k}}\sum_{\mu\in\Lambda^{R}_{\epsilon}}
m(\mu,\epsilon)\cot(\mu(Z)+\varphi_{\epsilon})\mu^{\#}.
\end{equation}
Here, we have used $\mu(Z)+\varphi_{\epsilon}\not\in\pi\mathbb{Z}$
for all $\mu\in\Lambda^{R}_{\epsilon}$ and $\epsilon\in C_{k}$.
Let $\xi(t):[0,T)\to C$
be an integral curve of the following ODE:
\begin{equation}\label{eqn:MCF_ODE_Hermann}
\xi'(t)=-\bs{X}_{\xi(t)}.
\end{equation}
We set $x(t):=\pi_{1}(\exp_{G} \xi(t))\in G/K_{1}$
and denote by $M_{x(t)}:=K_{2}\cdot x(t)$ the $K_{2}$-orbit in $G/K_{1}$ through $x(t)$.
Then, by \cite{Koike11},
$\{M_{x(t)}\}_{t\in[0,T)}$ gives the BMCF starting from $M_{x(0)}=:M$.
In addition, by \cite{LiuRa22},
we have $T=\infty$.
Let $H(t)$ and $A(t)$ denote the mean curvature
and the shape tensor of $M_{x(t)}$, respectively.
According to the facts stated in Section \ref{sec:MCF_equifocal},
the tangential focal data of the initial principal orbit $M=K_{2}\cdot x(0)$
plays an important role in describing the behavior of $H(t)$ and $A(t)$ along
the BMCF $\{M_{x(t)}\}_{t\in[0,\infty)}$.
Among the tangential focal data,
we have obtained $\{m_{1}^{F},m_{2}^{F}\}$ and $2g_{F}$
in terms of the symmetric triad $\mathcal{S}$ with multiplicities.
We can determine the length $l$ of the normal circle of $M_{x(0)}$
by using
its formula \eqref{eqn:l_formula_tildem},
from which $\delta_{F}$ and $\epsilon_{F}$ can be calculated.

\begin{ex}[No.~4 in Table \ref{table:CST_rank1}]
We consider the case of $(G,K_{1},K_{2})=(SU(n),S(U(1)\times U(n-1)),S(U(c)\times U(n-c)))$
with $2\leq c\leq n-2$.
Let $\kappa_{1}$ denotes the Einstein constant of $G/K_{1}=\mathbb{C}P^{n-1}$.
By \cite[Table 8]{BI25}
it is verified that
the corresponding symmetric triad $\mathcal{S}$ is of type $(\mbox{I-}BC_{1}\mbox{-}A_{1})$
with multiplicities $(m(\mu,1),m(\mu,-1))=(2(n-c-1),2(c-1))$, $(m(2\mu,1),m(2\mu,-1))=(1,0)$.
Let $\Gamma^{F}(M_{x(0)})=\{\theta_{F},l/(2g_{F}),\{m_{1}^{F},m_{2}^{F}\};2g_{F}\}$
denote the tangential focal data of $M_{x(0)}=K_{2}\cdot x(0)$.
From Lemma \ref{lem:focal_I-BC1-A1}
we have
$(m_{1}^{F},m_{2}^{F})=(2c-1,2n-2c-1)$
and $2g_{F}=2$.
By using $l_{\mu}=1$
(see Example \ref{ex:lmu_SUSUn-1SUcUd})
we get $l=\sqrt{2n}\,\pi/\sqrt{\kappa_{1}}$.
Hence we obtain
\begin{equation}
\delta_{F}=-\dfrac{n-2c}{n-1},
\quad
\epsilon_{F}=\dfrac{n}{4(n-1)\,\kappa_{1}}.
\end{equation}
\end{ex}

Theorems \ref{thm:mainthm2}--\ref{thm:main5}
are immediately from Facts \ref{thm:BMCF_evaluate}--\ref{thm:BMCF_evaluate3}
and Theorem \ref{thm:mainthm1}.
Here, we note that the constants $\alpha$, $\beta$, $\gamma$
and $\psi_{F}$ as in Table \ref{table:abg}
are given by calculating
\begin{equation}
\alpha=\dfrac{1}{\kappa\epsilon_{F}},\quad
\beta=\dfrac{m_{1}^{F}+m_{2}^{F}}{\kappa\epsilon_{F}(1-\delta_{F}^{2})},
\quad
\gamma=\dfrac{1-2\kappa\epsilon_{F}}{2\kappa\epsilon_{F}},
\quad
\psi_{F}=\dfrac{4\max\{m_{1}^{F},m_{2}^{F}\}}{(m_{1}^{F}+m_{2}^{F})^{2}}
\end{equation}
in terms of the tangential focal data of principal orbits.

\subsubsection*{Acknowledgements}
The first author would like to express his sincere gratitude to Professor Peter Quast
for the warm hospitality, continuous support, and stimulating discussions
during his stay at the University of Augsburg.
He would also like to express his sincere gratitude to Professor Osamu Ikawa
and Professor Shinji Ohno
for their valuable comments on symmetric triad with multiplicities corresponding to compact symmetric triads.

\end{document}